\documentclass{amsart}

\usepackage{url}
\usepackage[utf8]{inputenc}
\usepackage{mathrsfs}
\usepackage{amsmath,amssymb,amsfonts,amsthm}
\usepackage{enumerate}
\usepackage{amssymb}
\usepackage[all]{xy}
\usepackage{subcaption}
\usepackage[hidelinks]{hyperref}
\usepackage{latexsym}
\usepackage{tikz-cd}
\usepackage{comment}
\usepackage{yfonts} 
\usepackage{upgreek}
\usepackage{mathtools}
\usetikzlibrary{matrix, arrows}
\usepackage{rotating}
\usepackage{hyperref}
\usepackage{tabularx}

\newtheorem{thm}{Theorem}[section]
\newtheorem{lem}[thm]{Lemma}
\newtheorem{prop}[thm]{Proposition}
\newtheorem{cor}[thm]{Corollary}
\theoremstyle{definition}
\newtheorem{defin}[thm]{Definition}

\newtheorem{que}{Question}
\theoremstyle{remark}
\newtheorem{rmk}[thm]{Remark}

\numberwithin{equation}{section}

\newcommand{\bigslant}[2]{{\raisebox{.2em}{$#1$}\left/\raisebox{-.2em}{$#2$}\right.}}
\newcommand{\p}{\mathfrak{p}}
\newcommand{\OO}{\mathcal{O}}
\newcommand{\A}{\mathcal{A}}
\newcommand{\J}{\mathcal{J}}

\newcommand{\LL}{\mathcal{L}}
\newcommand{\HH}{\mathcal{H}}

\newcommand{\Mod}[1]{\ (\mathrm{mod}\ #1)}

\DeclareMathOperator{\Spec}{Spec}
\DeclareMathOperator{\Sp}{Sp}
\DeclareMathOperator{\Proj}{Proj}
\DeclareMathOperator{\Hom}{Hom}

\DeclareMathOperator{\EExt}{\mathscr{E}\textit{\kern -1pt {xt}}\,}

\DeclareMathOperator{\id}{id}
\DeclareMathOperator{\ord}{ord}
\DeclareMathOperator{\FJ}{FJ}

\DeclareMathOperator{\GL}{GL}

\DeclareMathOperator{\CH}{CH}

\DeclareMathOperator{\Sym}{Sym}
\DeclareMathOperator{\Ima}{Im}

\DeclareMathOperator{\tr}{tr}
\DeclareMathOperator{\rk}{rk}
\DeclareMathOperator{\FM}{FM}

\begin{document}

\title[Formal Fourier--Jacobi series from a cohomological point of view]{Modularity of formal Fourier--Jacobi series from a cohomological point of view}

\author{Marco Flores}
\email{}
\thanks{The author acknowledges support from the Deutsche Forschungsgemeinschaft (DFG, German Research Foundation) under Germany's Excellence Strategy – The Berlin Mathematics Research Center MATH+ (EXC-2046/1, project ID: 390685689).}

\begin{abstract}
We investigate the modularity of formal Fourier--Jacobi series by establishing cohomological vanishing results for line bundles defined on compactifications of $\mathcal{A}_g$. Working over $\mathbb{C}$, we show that the minimal compactification of $\mathcal{A}_2$ has only rational singularities, which allows us to characterize, for sufficiently large weights, the modularity of formal Fourier--Jacobi series of genus~$2$ via cohomological vanishing. Working over $\mathbb{Z}$, we introduce a level $n\geq 3$ structure and, via the the resolution morphism from a toroidal compactification of $\mathcal{A}_{g,n}$ to its minimal compactification, we characterize the modularity of arithmetic formal Fourier--Jacobi series with a level $n\geq 3$ structure and sufficiently large weight via cohomological vanishing.
\end{abstract}

\maketitle
\tableofcontents

\section{Introduction}
\subsection{Background}

In recent decades, the theory of modular forms has occupied a prominent position in mathematics, as it has proven to be intimately related to other, seemingly separate, mathematical fields. For instance, modular forms have been a key tool in establishing notorious number-theoretic results, such as Fermat's Last Theorem \cite{Wil95}, or the $290$-Theorem on universal quadratic forms \cite{BH05}. On the other hand, they have also been utilized to prove major geometric results, such as the optimality of the $E_8$-lattice sphere packing in dimension~$8$ \cite{Via17}, and the optimality of the Leech lattice sphere packing in dimension~$24$ \cite{CKMRV17}. 

Much of the importance of modular forms comes in fact from their Fourier expansions, or, in the higher dimensional case, their Fourier--Jacobi expansions. As it turns out, there are many examples of modular forms, whose Fourier and Fourier--Jacobi coefficients contain information which is relevant to a given arithmetic or geometric problem. It then becomes important to determine when a given class of mathematical objects has a relationship of this kind with modular forms, and thus speak of \textit{modularity conjectures} or \textit{modularity results}, the most celebrated of which is perhaps the modularity theorem for semistable elliptic curves over $\mathbb{Q}$, proved by Wiles in \cite{Wil95}. 

A further example of a modularity result is the one investigated by Hirzebruch and Zagier in \cite{HZ76}, where they showed that the intersection numbers of certain cycles on Hilbert modular surfaces are in fact equal to the Fourier coefficients of a suitable modular form. This modularity result was generalized by Kudla and Millson in the work \cite{KM90}, to the case of intersection numbers of cycles on locally symmetric spaces $X$ associated to orthogonal and unitary groups, by constructing liftings from the cohomology with compact supports of $X$ to a space of modular forms. Furthermore, the case where $X$ is associated to the orthogonal group $O(n,2)$ was studied in detail in \cite{Kud97}. There, Kudla associated to each positive semi-definite symmetric matrix $N\in\Sym_g(\mathbb{Q})$, a rational equivalence class $Z(N)\in\CH^g(X)$ of special cycles of codimension $g$, and considered their generating series \begin{equation}\label{kudlagenseries}
    A_g(\tau)=\sum_{\substack{0\leq N\in\Sym_g(\mathbb{Q})}}Z(N)e^{2\pi i\tr(N\tau)},
\end{equation} where the variable $\tau$ belongs to the Siegel upper half-space $\HH_g$. The results in \cite{KM90} then imply that the image in cohomology of this generating series is in fact a modular form, which led Kudla to ask in \cite{Kud97} whether already the series \eqref{kudlagenseries} at the level of Chow groups is a modular form.

This modularity conjecture of Kudla was positively answered by Bruinier and Raum in \cite{BR15}. A major, previously established, reduction step in the proof of Kudla's conjecture, due to Zhang \cite{Zha09}, consists in showing that the coefficients of the generating series \eqref{kudlagenseries} are in fact Jacobi forms. Bruinier and Raum then proved that any symmetric formal Fourier--Jacobi series, that is, a formal series whose coefficients are Jacobi forms and are invariant with respect to a natural $\GL_g(\mathbb{Z})$-action, converges and thus equals the Fourier--Jacobi expansion of a modular form. This modularity result for formal Fourier--Jacobi series was extended to the arithmetic setting over $\mathbb{Z}$ by Kramer \cite{Kra22}, with the use of the theory of arithmetic compactifications of the moduli stack $\mathcal{A}_g$ of principally polarized abelian schemes $A\to S$ of relative dimension~$g$, where $S$ is some base scheme.

The modularity result of Bruinier--Raum can be stated as follows. Let $M_k(\Gamma_g)$ denote the $\mathbb{C}$-vector space of Siegel modular forms of weight $k$ and genus $g$ with respect to the Siegel modular group $\Gamma_g$, and let $\FM_k^{(g,1)}$ denote the vector space of symmetric formal Fourier--Jacobi series of weight $k$, genus $g$ and cogenus $1$ (see Definition \ref{formalsymdef}). Then, the main result in \cite{BR15} states that the Fourier--Jacobi homomorphism \[\FJ\colon M_k(\Gamma_g)\to\FM_k^{(g,1)},\] taking a Siegel modular form to its Fourier--Jacobi expansion of cogenus $1$, is surjective.

A key idea, suggested in \cite{Kra22}, is that the Fourier--Jacobi homomorphism can be given a cohomological interpretation, while furthermore working over $\mathbb{Z}$, as follows. Let $\overline{\mathcal{A}}_g$ be a toroidal compactification of the moduli stack $\mathcal{A}_g$, let $\omega$ denote the Hodge line bundle on $\overline{\mathcal{A}}_g$, and let $\J$ be a sheaf of ideals supported on the boundary $\overline{\mathcal{A}}_g\setminus\mathcal{A}_g$. Then, the space of Siegel modular forms of weight~$k$ is given by the space of global sections $H^0(\overline{\mathcal{A}}_g, \omega^k)$, and the Fourier--Jacobi homomorphism is recovered as the map 
 \[\FJ\colon H^0(\overline{\mathcal{A}}_g,\omega^k)\to \varprojlim_m H^0(\overline{\mathcal{A}}_g, \omega^k\otimes\bigslant{\OO_{\overline{\mathcal{A}}_g}}{\mathcal{J}^m}),\]
which is found in the long exact sequence 

\begin{multline}\label{lesintr}0\to \varprojlim_m H^0(\overline{\mathcal{A}}_g,\omega^k\otimes\mathcal{J}^m) \to H^0(\overline{\mathcal{A}}_g,\omega^k)\xrightarrow{\FJ} \varprojlim_m H^0(\overline{\mathcal{A}}_g, \omega^k\otimes\bigslant{\OO_{\overline{\mathcal{A}}_g}}{\mathcal{J}^m})\\
\to  \varprojlim_m H^1(\overline{\mathcal{A}}_g,\omega^k\otimes\mathcal{J}^m)\to H^1(\overline{\mathcal{A}}_g,\omega^k)\to\cdots\end{multline} of abelian groups. The surjectivity of $\FJ$, together with the exactness of the sequence~\eqref{lesintr}, implies that the homomorphism \[\varprojlim_m H^0(\overline{\mathcal{A}}_g, \omega^k\otimes\bigslant{\OO_{\overline{\mathcal{A}}_g}}{\mathcal{J}^m})\to\varprojlim_m H^1(\overline{\mathcal{A}}_g,\omega^k\otimes\mathcal{J}^m)\]
 equals the zero homomorphism. This begs the question of whether already the abelian group $\varprojlim_m H^1(\overline{\mathcal{A}}_g,\omega^k\otimes\mathcal{J}^m)$ itself is equal to zero, which we pose as follows.

 \begin{que}\label{q1}
Given $k\in\mathbb{N}$, does the cohomological vanishing
\begin{equation*}\label{vanlim}
    \varprojlim_m H^1(\overline{\mathcal{A}}_g,\omega^k\otimes\mathcal{J}^m)=0
\end{equation*}
hold?
\end{que}

An argument which answers Question~\ref{q1} positively would thus provide an alternative proof of the modularity result of Bruinier--Raum when working over~$\mathbb{C}$, as well as an alternative proof of Kramer's result when working over $\mathbb{Z}$. The main guiding principle of this paper is then to investigate Question~\ref{q1}, as well as its following refinement.

\begin{que}\label{q2}
  For which pairs $(k,m)\in\mathbb{N}^2$ does the cohomological vanishing
  \begin{equation*}\label{vanfin}
      H^1(\overline{\mathcal{A}}_g,\omega^k\otimes\mathcal{J}^m)=0
  \end{equation*} hold?
\end{que}

\subsection{Organization and main results}
Sections~\ref{chapprep} and \ref{ag_ref} are of a preparatory nature, their main objective being to lay out notation and recall necessary concepts and tools. They provide no original results. In them, we recall the basic notions in the theory of Siegel modular forms, with an emphasis on Fourier--Jacobi expansions, first in the classical analytic setting over $\mathbb{C}$ and then in the arithmetic setting over $\mathbb{Z}$. In particular, when it comes to the arithmetic setting in Section~\ref{ag_ref}, we give a broad overview of the theory of arithmetic compactifications of the moduli stack $\mathcal{A}_g$, as introduced by Faltings--Chai in \cite{FC90}. This constitutes the core technical language that we use throughout this paper. We end Section \ref{ag_ref} in Subsection~\ref{seccohref} by explaining in detail, with the aid of Grothendieck's existence theorem on formal functions, how the modularity problem for formal Fourier--Jacobi series can be given a cohomological interpretation.

In Section~\ref{chc} we work over $\mathbb{C}$. We begin in Subsection~\ref{projinv} by recalling how one constructs the quotient of a projective scheme by a finite group of automorphisms. In Subsection~\ref{ratsingg2}, we exhibit the minimal compactification of $\mathcal{A}_2$, denoted $\mathcal{A}^*_2$, as the quotient of a smooth projective variety by a finite group of automorphisms, from which we deduce that $\mathcal{A}^*_2$ has only rational singularities. As an interesting application of this fact, we observe in Corollary~\ref{surjeq2} that, in the case $g=2$, a positive answer to Question~\ref{q1} is equivalent to the surjectivity of the Fourier--Jacobi homomorphism.

In Section~\ref{chz} we work over $\mathbb{Z}$, and we consider the moduli stack $\mathcal{A}_{g,n}$ of principally polarized abelian schemes $A\to S$ of relative dimension $g$, where $S$ is some base scheme, with a principal level structure $n\geq 3$. In Subsection~\ref{secvalk} we prove a cohomological vanishing result (Theorem~\ref{largek}) for sufficiently large weights, thus obtaining a partial answer to Question~\ref{q2} in the level $n\geq 3$ case. This allows us to show, in Subsection~\ref{characz}, that for sufficiently large $k$ and with a level $n\geq 3$ structure, a positive answer to Question~\ref{q1} is equivalent to the surjectivity of the Fourier--Jacobi homomorphism (Theorem~\ref{surjeq}). The idea of the proof is to compare the sheaf cohomology groups on a toroidal compactification $\overline{\mathcal{A}}_{g,n}$, with sheaf cohomology groups on the minimal compactification $\mathcal{A}^*_{g,n}$, and then use Serre's cohomological criterion of ampleness.
\subsection{Acknowledgements}
We deeply thank Jürg Kramer, for his suggestion to investigate the matters present in this paper, as well as for numerous useful and illuminating discussions. We also thank Jan Bruinier, Thomas Krämer and Gari Peralta, for their helpful observations.

\section{Siegel modular forms over \texorpdfstring{$\mathbb{C}$}{C}}\label{chapprep}
In the present section, we recall the basic definitions in the classical theories of Siegel modular forms and Jacobi forms over $\mathbb{C}$, following \cite{BGHZ08} and \cite{Zie89}. We then explain the result in \cite{BR15} regarding the modularity of symmetric formal Fourier--Jacobi series.

\subsection{Siegel modular forms}\label{siegel}
If $R$ denotes a ring and $n$ is a natural number, we will denote the set of $n\times n$ matrices with coefficients in $R$ by $M_n(R)$, and we we will denote the set of $n\times n$ symmetric matrices with coefficients in $R$ by $\Sym_n(R)$.

Let us fix an integer $g\geq 1$. Let $e_1,\dots,e_g,f_1,\dots,f_g$ be the canonical basis of the lattice $\mathbb{Z}^{2g}$, which we endow with the standard symplectic form $\langle\cdot\,,\cdot\rangle$ given by 
\[\langle e_i,e_j \rangle=0, \langle  f_i,f_j\rangle=0, \text{ and } \langle e_i,f_j \rangle=\delta_{ij}\]
for $i,j=1,\dots,g$, where $\delta_{ij}$ denotes Kronecker's delta. The symplectic group $\Sp_{2g}(\mathbb{Z})$ is then defined as the automorphism group of the symplectic lattice $(\mathbb{Z}^{2g},\langle\cdot\,,\cdot\rangle)$. More explicitly, this group is given by \[\Sp_{2g}(\mathbb{Z})=\Bigg\{   \begin{pmatrix}
    A &B \\
    C &D
    \end{pmatrix}\in M_{2g}(\mathbb{Z})\colon  AB^t=BA^t, CD^t=DC^t, AD^t-BC^t=1_g \Bigg\},\] where $1_g$ denotes the $g\times g$ identity matrix. We refer to $\Gamma_g\coloneqq\Sp_{2g}(\mathbb{Z})$ as the \textit{Siegel modular group}. 

    The \textit{Siegel upper half-space} $\HH_g$ of degree $g$ is defined as the set of $g\times g$ symmetric matrices over $\mathbb{C}$ whose imaginary part is positive definite, namely, \[\HH_g\coloneqq\{
Z\in\Sym_g(\mathbb{C})\mid \Ima(Z)>0\}.\]  We then consider the action of $\Gamma_g$ on $\HH_g$ given by fractional linear transformations, that is, \[\gamma\cdot Z\coloneqq(AZ+B)(CZ+D)^{-1}\] for each $Z\in\HH_g$ and $\gamma=\begin{psmallmatrix}
    A &B \\
    C &D
    \end{psmallmatrix}\in\Gamma_g$. Given an integer $k\in\mathbb{Z}$, a holomorphic function $f\colon \HH_g\to\mathbb{C}$ is called a \textit{Siegel modular form of genus} $g$ \textit{and weight} $k$, 
    if
    \begin{enumerate}
        \item[(i)] it satisfies the functional equation \[f(\gamma\cdot Z)=\det(CZ+D)^kf(Z)\] for all $Z\in\HH_g$, $\gamma=\begin{psmallmatrix}
    A &B \\
    C &D
    \end{psmallmatrix}\in\Gamma_g$, and
    \item[(ii)] it has a Fourier expansion of the form

    \[f(Z)= \sum_{\substack{0\leq N\in\Sym_g(\mathbb{Q})\\N \text{ half-integral}}}c(N)e^{2\pi i\tr(NZ)}\]
    with $c(N)\in\mathbb{C}$. 
    \end{enumerate}  Here, a symmetric matrix $N\in \Sym_g(\mathbb{Q})$ is called $\textit{half-integral}$ if $2N\in\Sym_g(\mathbb{Z})$ and its diagonal entries are even. We point out that, when $g\geq 2$, condition (ii) follows from condition (i) due to Koecher's principle \cite[Satz 2]{Koe54}.

    We note here that, by thinking of the elements of $\mathbb{Z}^g$ as column vectors, each $Z\in\HH_g$ determines a $g$-dimensional principally polarized abelian variety $A_Z=\mathbb{C}^g/(\mathbb{Z}^g+Z\mathbb{Z}^g)$, and $A_Z\cong A_{Z'}$ if and only if $Z'=\gamma\cdot Z$ for some $\gamma\in\Gamma_g$. In fact, the orbit space $\Gamma_g\backslash\HH_g$ can be interpreted as the moduli space of principally polarized abelian varieties of dimension $g$. Later on, when we pass to the arithmetic-geometric setting in Section~\ref{ag_ref}, we will recall how this allows one to give Siegel modular forms a geometric interpretation.
    
    Siegel modular forms of weight $k$ constitute a finite dimensional $\mathbb{C}$-vector space, which we denote by $M_k({\Gamma_g})$. Moreover, the product of two Siegel modular forms of weight $k$ and $k'$, respectively, is a Siegel modular form of weight $k+k',$ giving \[M_\bullet(\Gamma_g)\coloneqq\bigoplus_{k\in\mathbb{Z}}M_k(\Gamma_g)\] the structure of a graded $\mathbb{C}$-algebra, which is known to be finitely generated \cite[Theorem~10.14]{BB66}.

    The Fourier expansion of a Siegel modular form $f\in M_k(\Gamma_g)$ allows us to define its \textit{vanishing order}, as follows. We say that a half-integral symmetric matrix $N\in \Sym_g(\mathbb{Q})$ \textit{represents} a rational number $o\in\mathbb{Q}$, if there exists $v\in\mathbb{Z}^g$ such that $v^tNv=o$. In case $f\neq 0$, we then define its vanishing order as \[\ord(f)\coloneqq\inf \{o\in\mathbb{Q} \mid \exists N \textnormal{ representing } o \textnormal{ such that } c(N)\neq 0\},\] and, if $f=0$, we put $\ord(f)\coloneqq \infty$. For each $o\in\mathbb{Q}$ we let \[M_k(\Gamma_g)[o]\coloneqq\{f\in M_k(\Gamma_g)\mid\ord(f)\geq o\},\] which is a vector subspace of $M_k(\Gamma_g)$. A classical result (see, e.g., \cite[Proposition~1.4]{BR15}) states that $M_k(\Gamma_g)[o]=0$ for all sufficiently large $o$, depending on $g$ and $k$.

Now, let $f\in M_k(\Gamma_g)$ be a Siegel modular form and let $0\leq l\leq g$ be an integer. If we write an element $Z\in\HH_g$ as \[Z=\begin{pmatrix}
    \tau &z^t \\
    z &\tau'
    \end{pmatrix}\]
with $\tau\in\HH_{g-l},\tau'\in\HH_{l}$ and $z\in\mathbb{C}^{l\times (g-l)}$, we notice that $f$ is invariant under the linear transformation $\tau'\mapsto\tau'+1_l$. Indeed, this follows from the definition of a Siegel modular form applied to the matrix \[\begin{pmatrix}
    1_g &B \\
    0 &1_g
    \end{pmatrix}\in\Gamma_g,\] where $B=\begin{psmallmatrix}
    0 &0 \\
    0 &1_l
    \end{psmallmatrix}.$ We thus obtain a partial Fourier expansion of $f$ with respect to the variable $\tau'$, given by \[f(Z)=\sum_{\substack{0\leq m\in\Sym_l(\mathbb{Q})\\m \text{ half-integral}}}f_m(\tau,z)e^{2\pi i\tr(m\tau')},\]
where each function $f_m\colon \HH_{g-l}\times\mathbb{C}^{l\times (g-l)}\to\mathbb{C}$ is holomorphic. This is called the \textit{Fourier--Jacobi expansion of cogenus} $l$ \textit{of} $f$. The functions $f_m$ inherit some functional equations from the ones satisfied by $f$, which will motivate Definition~\ref{Jacdef} below.

\subsection{Jacobi forms}\label{subsecjacobic}

Let us consider the discrete Heisenberg group \[H_\mathbb{Z}^{(g-l,l)}\coloneqq\{(\lambda,\mu,\chi)\in\mathbb{Z}^{l\times(g-l)}\times\mathbb{Z}^{l\times(g-l)}\times\mathbb{Z}^{l\times l}\mid\chi+\mu\lambda^t \text{ is symmetric}\},\] whose group operation is given by 
\[(\lambda,\mu,\chi)\circ(\lambda',\mu',\chi')=(\lambda+\lambda',\mu+\mu',\chi+\chi'+\lambda\mu'^t-\mu\lambda'^t).\] Notice that $0=(0,0,0)$ is its neutral element. 

We can embed the discrete Heisenberg group into the Siegel modular group~$\Gamma_g$, via the map given by the assignment
\[(\lambda,\mu,\chi)\mapsto\begin{pmatrix}
    1_{g-l} &0 &0 &\mu^t \\
    \lambda &1_l &\mu &\chi \\
    0 &0 &1_{g-l} &-\lambda^t \\
    0 &0 &0 &1_l
    \end{pmatrix}.\] We can also embed $\Gamma_{g-l}$ into $\Gamma_{g}$, via the map given by the assignment

    \[\begin{pmatrix}
    A &B \\
    C &D
    \end{pmatrix}\mapsto\begin{pmatrix}
    A &0 &B &0 \\
    0 &1_l &0 &0 \\
    C &0 &D &0 \\
    0 &0 &0 &1_l
    \end{pmatrix}.\] These two subgroups of $\Gamma_g$ have trivial intersection; the \textit{Jacobi group} is thus defined as the semidirect product $G_\mathbb{Z}^{(g-l,l)}\coloneqq\Gamma_{g-l}\ltimes H_\mathbb{Z}^{(g-l,l)}$. We then consider the action of $G_\mathbb{Z}^{(g-l,l)}$ on $\HH_{g-l}\times\mathbb{C}^{l\times (g-l)}$, given by
    \[(\gamma,\zeta)\bullet (\tau,z)\coloneqq (\gamma\cdot\tau,(z+\lambda\tau+\mu)(C\tau+D)^{-1})\] for each $(\tau, z)\in\HH_{g-l}\times\mathbb{C}^{l\times (g-l)}$, $\gamma=\begin{psmallmatrix}
    A &B \\
    C &D
    \end{psmallmatrix}\in\Gamma_{g-l}$ and $\zeta=(\lambda,\mu,\chi)\in H_\mathbb{Z}^{(g-l,l)}$. 

    \begin{defin}\label{Jacdef}
        Let $k$ be an integer and let $m\in\Sym_l(\mathbb{Q})$ be a positive semi-definite, half-integral symmetric matrix. A holomorphic function \[\phi\colon \HH_{g-l}\times\mathbb{C}^{l\times (g-l)}\to\mathbb{C}\] is called a \textit{Jacobi form of genus} $g-l$, \textit{weight} $k$ \textit{and index} $m$ if
        \begin{enumerate}
            \item [(i)] it satisfies the functional equation \[\phi((\gamma,0)\bullet(\tau,z))=\det(C\tau+D)^ke^{2\pi i\tr(mz(C\tau+D)^{-1}Cz^t)}\phi(\tau,z)\] for all $\gamma=\begin{psmallmatrix}
    A &B \\
    C &D
    \end{psmallmatrix}\in\Gamma_{g-l}$ and $(\tau,z)\in\HH_{g-l}\times\mathbb{C}^{l\times (g-l)}$,
            \item[(ii)] it satisfies the functional equation \[\phi((1,\zeta)\bullet(\tau,z))=e^{-2\pi i\tr(m(\lambda\tau\lambda^t+2\lambda z^t+\chi+\mu\lambda^t))}\phi(\tau,z)\] for all $\zeta=(\lambda,\mu,\chi)\in H_\mathbb{Z}^{(g-l,l)}$ and $(\tau,z)\in\HH_{g-l}\times\mathbb{C}^{l\times (g-l)}$, and
            \item[(iii)] it has a Fourier expansion of the form \[\phi(\tau,z)=\sum_{\substack{0\leq n\in\Sym_{g-l}(\mathbb{Q})\\n \text{ half-integral}}}\sum_{r\in\mathbb{Z}^{(g-l)\times l}}c(n,r)e^{2\pi i\tr(s^{-1}n\tau+rz)}\] for a suitable $0\neq s\in\mathbb{Z}$, such that $c(n,r)\neq 0$ only if the matrix \[\begin{pmatrix}
    s^{-1}n &\frac{1}{2}r \\
    \frac{1}{2}r^t &m
    \end{pmatrix}\] is positive semi-definite.
        \end{enumerate}
    \end{defin}
\noindent As was the case for Siegel modular forms, there is a Koecher principle \cite[Lemma~1.6]{Zie89} for Jacobi forms; namely, if $g-l\geq 2$, then the conditions (i) and (ii) imply condition (iii).

Jacobi forms of genus $g-l$, weight $k$ and index $m$ constitute a finite dimensional $\mathbb{C}$-vector space \cite[Theorem~1.8]{Zie89}, which we denote by $J_{k,m}(\Gamma_{g-l})$. Moreover, the product of two Jacobi forms $\phi\in J_{k,m}(\Gamma_{g-l})$ and $\phi'\in J_{k',m'}(\Gamma_{g-l})$ is a Jacobi form of weight $k+k'$ and index $m+m'$, giving \[J_{\bullet,\bullet}(\Gamma_{g-l})\coloneqq\bigoplus_{k,m}J_{k,m}(\Gamma_{g-l})\] the structure of a bigraded $\mathbb{C}$-algebra. 

\begin{rmk}
    The $\mathbb{C}$-algebra $J_{\bullet,\bullet}(\Gamma_{g-l})$ is known not to be finitely generated. Moreover, upon fixing $k,m>0$, even the subalgebra $\bigoplus_{s}J_{sk,sm}(\Gamma_{g-l})$ is not finitely generated; this is the main theorem in \cite{BBHJ22}.
\end{rmk}

As we have done for Siegel modular forms, we can also define the \textit{vanishing order of} a Jacobi form $\phi\in J_{k,m}(\Gamma_{g-l})$ via its Fourier expansion. Namely, if $\phi\neq 0$, we define \[\ord(\phi)\coloneqq\inf\{o\in\mathbb{Q}\mid\exists n,r \textnormal{ such that } c(n,r)\neq 0 \textnormal{ with } n \textnormal{ representing }o\}\] and, if $\phi=0$, we put $\ord(\phi)\coloneqq\infty$. For each $o\in\mathbb{Q}$ we let \[J_{k,m}(\Gamma_{g-l})[o]\coloneqq\{\phi\in J_{k,m}(\Gamma_{g-l})\mid\ord(\phi)\geq o\},\] which is a vector subspace of $J_{k,m}(\Gamma_{g-l})$. In the case that $l=1$, it can be shown that $J_{k,m}(\Gamma_{g-1})[m]=0$ for all sufficiently large $m$, depending on $g$ and $k$. More precisely, if \[\varrho(f)\coloneqq\frac{k}{\ord(f)}\] denotes the so-called \textit{slope of} a non-zero Siegel modular form $f\in M_k(\Gamma_g)$, and we put \[\varrho_g\coloneqq\inf_{f\in M_\bullet(\Gamma_g)\setminus\{0\}}\varrho(f),\] then one has the following result.

\begin{prop}\label{jacvanord}
    If $m\in\mathbb{N}$ satisfies $m>\frac{4}{3}\frac{k}{\varrho_{g-1}}$, then \[J_{k,m}(\Gamma_{g-1})[m]=0.\]
\end{prop}

\begin{proof}
    See, e.g., \cite[Lemma~2.7]{BR15}.
\end{proof}

\subsection{Symmetric formal Fourier--Jacobi series}\label{sffjs}

In their paper \cite{BR15}, Bruinier and Raum determine the kind of formal series of Jacobi forms which arise as the Fourier--Jacobi expansion of a Siegel modular form. A necessary condition, which turns out to be sufficient, is obtained by noticing that the Fourier coefficients of a Siegel modular form satisfy the following symmetry condition.

\begin{prop}\label{SymmetryCondition}
    Let $f\in M_k(\Gamma_g)$ be a Siegel modular form with Fourier expansion \[f(Z)= \sum_{\substack{0\leq N\in\Sym_g(\mathbb{Q})\\N \text{ half-integral}}}c(N)e^{2\pi i\tr(NZ)}.\] Then $c(u^tNu)=\det(u)^kc(N)$ for each half-integral $0\leq N\in\Sym_g(\mathbb{Q})$ and for every $u\in\GL_g(\mathbb{Z})$. 
\end{prop}

\begin{proof}
    We notice that $\GL_g(\mathbb{Z})$ acts on the index set \[\{N\in \Sym_g(\mathbb{Q})\mid N\geq 0, N \text{ half-integral}\}\] via $N\mapsto u^tNu$. Hence, we have
\begin{align*}
    f(uZu^t)&= \sum_{\substack{0\leq N\in\Sym_g(\mathbb{Q})\\N \text{ half-integral}}}c(N)e^{2\pi i\tr(NuZu^t)}\\
    &=\sum_{\substack{0\leq N\in\Sym_g(\mathbb{Q})\\N \text{ half-integral}}}c(N)e^{2\pi i\tr(u^tNuZ)}\\
    &=\sum_{\substack{0\leq N\in\Sym_g(\mathbb{Q})\\N \text{ half-integral}}}c(u^tNu)e^{2\pi i\tr(NZ)}.
\end{align*} Furthermore, if for each $u\in\GL_g(\mathbb{Z})$ we let \[\gamma_u\coloneqq\begin{pmatrix}
    u & 0\\
    0 &(u^{-1})^t
    \end{pmatrix}\in \Gamma_g,\] then the functional equation of the Siegel modular form $f$ applied to the matrix $\gamma_u$ gives us $f(uZu^t)=f(\gamma_u\cdot z)=\det(u)^kf(Z)$. Therefore, a comparison of the Fourier expansions yields $c(u^tNu)=\det(u)^kc(N)$ for each $N$ and for every $u\in\GL_g(\mathbb{Z})$.
\end{proof}

If the Fourier--Jacobi expansion of cogenus $l$ of the Siegel modular form $f\in M_k(\Gamma_g)$ is given by \begin{equation}\label{FourierJacobiExpansion}
    f(Z)=\sum_{\substack{0\leq m\in\Sym_l(\mathbb{Q})\\m \text{ half-integral}}}f_m(\tau,z)e^{2\pi i\tr(m\tau')},
\end{equation} then the functional equations satisfied by $f$ induce on each $f_m$ the transformation behavior of a Jacobi form of genus $g-l$, weight $k$ and index $m$, that is, $f_m\in J_{k,m}(\Gamma_{g-l})$. Each Jacobi form $f_m$ has a Fourier expansion of the form 
\begin{equation*}
    f_m(\tau,z)=\sum_{\substack{0\leq n\in\Sym_{g-l}(\mathbb{Q})\\n \text{ half-integral}}}\sum_{r\in\mathbb{Z}^{(g-l)\times l}}c_m(n,r)e^{2\pi i\tr(n\tau+rz)}
\end{equation*}
 and, if we substitute each of these into (\ref{FourierJacobiExpansion}), we get back the Fourier expansion \[f(Z)= \sum_{\substack{0\leq N\in\Sym_g(\mathbb{Q})\\N \text{ half-integral}}}c(N)e^{2\pi i\tr(NZ)},\] where $c(N)=c_m(n,r)$ for $N=\begin{psmallmatrix}
    n &\frac{1}{2}r \\
    \frac{1}{2}r^t &m
    \end{psmallmatrix}$. Thus, the symmetry condition in the Fourier coefficients of $f$ from Proposition~\ref{SymmetryCondition} translates into a symmetry condition of the Fourier coefficients of each $f_m$.

    \begin{defin}\label{formalsymdef}
        Let \[f(Z)=\sum_{\substack{0\leq m\in\Sym_l(\mathbb{Q})\\m \text{ half-integral}}}f_m(\tau,z)e^{2\pi i\tr(m\tau')}\] be a formal series with $f_m\in J_{k,m}(\Gamma_{g-l})$ for each $m$. We define its Fourier coefficients as \[c(N)\coloneqq c_m(n,r), \ \ \ \ N=\begin{pmatrix}
    n &\frac{1}{2}r \\
    \frac{1}{2}r^t &m
    \end{pmatrix},\] where $c_m(n,r)$ are the Fourier coefficients of the Jacobi form $f_m$. If the equality \[c(u^tNu)=\det(u)^kc(N)\] holds for every half-integral $0\leq N\in\Sym_g(\mathbb{Q})$ and for every $u\in\GL_g(\mathbb{Z})$, we say that $f$ is a \textit{symmetric} formal Fourier--Jacobi series of genus $g$, cogenus $l$ and weight $k$. 
    \end{defin}

    The symmetric Fourier--Jacobi series of genus $g$, cogenus $l$ and weight $k$ constitute a $\mathbb{C}$-vector space, which we denote by $\FM_k^{(g,l)}$. According to \cite[Proposition~2.2]{BR15}, the graded ring \[\FM_\bullet^{(g,l)}\coloneqq\bigoplus_{k\in\mathbb{Z}}\FM_k^{(g,l)}\] is an algebra over $M_\bullet(\Gamma_g)$.

    The vanishing order of a symmetric formal Fourier--Jacobi series $f\in\FM^{(g,1)}_k$ of cogenus $1$ is defined in the same way as that of a Siegel modular form, namely, as
    \[\ord(f)\coloneqq\inf \{m\in\mathbb{Q} \mid \exists N \textnormal{ representing } m \textnormal{ such that } c(N)\neq 0\}\] in case $f\neq 0$, and, if $f=0$, we put $\ord(f)\coloneqq \infty$. For each $m\in\mathbb{Q}$ we then let
    \[\FM^{(g,1)}_k[m]\coloneqq\{f\in\FM^{(g,1)}_k\mid\ord(f)\geq m\}.\] As a consequence to Proposition~\ref{jacvanord}, we have the following result.

    \begin{cor}\label{vanordfm}
         If $m\in\mathbb{N}$ satisfies $m>\frac{4}{3}\frac{k}{\varrho_{g-1}}$, then \[\FM^{(g,1)}_k[m]=0.\]
    \end{cor}

    \begin{proof}
        For every $m\in\mathbb{N}$, the map \[\FM^{(g,1)}_k[m]\to J_{k,m}(\Gamma_{g-1})[m] \] assigning to a symmetric formal Fourier--Jacobi series its $m$-th Fourier--Jacobi coefficient, has kernel equal to $\FM^{(g,1)}_k[m+1]$. Hence, according to Lemma~\ref{jacvanord}, we have \[\FM^{(g,1)}_k[m]=\FM^{(g,1)}_k[m']\] for every $m,m'>\frac{4}{3}\frac{k}{\varrho_{g-1}}$, which implies that \[\FM^{(g,1)}_k[m]=\bigcap_{m\in\mathbb{N}}\FM^{(g,1)}_k[m]=0\] for every $m>\frac{4}{3}\frac{k}{\varrho_{g-1}}$.
    \end{proof}

    The main theorem in \cite{BR15} is \cite[Theorem~4.1]{BR15}, which states that \[\FM_\bullet^{(g,1)}=M_\bullet(\Gamma_g).\] Furthermore, this modularity result is extended to more general Siegel modular forms of half-integer weight and to every cogenus $0<l<g$ in \cite[Theorem~4.5]{BR15}. The proof of \cite[Theorem~4.1]{BR15} may be broadly divided into three parts:
\begin{enumerate}[(1)]
    \item\label{br1} Show that $\FM_\bullet^{(g,1)}$ is finitely generated as a graded $M_\bullet(\Gamma_g)$-module.
    \item\label{br2} Show that if $f\in\FM_\bullet^{(g,1)}$ satisfies a non-trivial algebraic relation over $M_\bullet(\Gamma_g)$, then $f$ converges absolutely in an open neighborhood of the boundary divisor $\overline{\mathcal{A}}_g\setminus\mathcal{A}_g$.
    \item\label{br3} Show that such $f\in\FM_\bullet^{(g,1)}$ in fact admits an analytic continuation to the whole $\mathcal{H}_g$, and hence $f\in M_\bullet(\Gamma_g)$. 
\end{enumerate}
Here, $\overline{\mathcal{A}}_g$ denotes a toroidal compactification of the moduli space $\mathcal{A}_g=\Gamma_g\backslash \HH_g$ of principally polarized abelian varieties of dimension $g$.

In order to establish \eqref{br1}, the authors bound the asymptotic dimension of $\FM_k^{(g,1)}$ by embedding it into a product of spaces of Jacobi forms with bounded vanishing order. For the proof of \eqref{br2}, they localize at each boundary point, and deduce the convergence around that point from the fact that the local ring  at each boundary point is algebraically closed in its completion. Then, they are able to prove \eqref{br3} by employing the fact that the Picard group of the minimal compactification $\mathcal{A}_g^*$ is generated, up to torsion, by the class of the so-called \textit{Hodge line bundle}. The proof is then finished by observing that part \eqref{br1} implies that every $f\in\FM_\bullet^{(g,1)}$ satisfies a non-trivial algebraic relation over $M_\bullet(\Gamma_g)$.

We postpone the precise definitions of the minimal and toroidal compactifications of $\mathcal{A}_g$, as well as that of the Hodge line bundle, to the next section. Our reason is the following. The minimal compactification $\mathcal{A}_g^*$, also known as the Satake--Baily--Borel compactification, was first constructed in \cite{Sat56} (and, in the more general setting of locally symmetric domains, in \cite{BB66}), while toroidal compactifications were first constructed in \cite{AMRT75}. All of this is done in the analytic setting over $\mathbb{C}$. However, our intended goal is to work over $\mathbb{Z}$. Therefore, in Section~\ref{ag_ref}, we will recall the constructions of these compactifications as performed by Faltings--Chai in the arithmetic setting over $\mathbb{Z}$. Siegel modular forms over $\mathbb{Z}$ are then defined as global sections of tensor powers of the Hodge line bundle on a toroidal compactification $\overline{\mathcal{A}}_g$, and the analytic picture that has been discussed so far is recovered by base changing to $\mathbb{C}$.
\section{Siegel modular forms over \texorpdfstring{$\mathbb{Z}$}{Z}}\label{ag_ref}
In this section, we recall the construction carried out in \cite{FC90} of toroidal compactifications and the minimal compactification of the moduli stack $\mathcal{A}_g$ of principally polarized abelian schemes $A\to S$ of relative dimension $g$, where $S$ is some base scheme. This allows one to define Siegel modular forms, Jacobi forms and their various expansions over~$\mathbb{Z}$. Our reference for Jacobi forms is \cite{Kra95}. We then explain the main result in \cite{Kra22} on the modularity of formal Fourier--Jacobi series over~$\mathbb{Z}$, analogous to that of Bruinier--Raum over $\mathbb{C}$. Finally, in Subsection~\ref{seccohref}, we explain how the Fourier--Jacobi homomorphism $\FJ$, which assigns to a Siegel modular form its Fourier--Jacobi expansion of cogenus $1$, can be located within a long exact sequence of (inverse limits of) cohomology groups. It then becomes natural to inquire on the surjectivity of $\FJ$ by asking the stronger question, of whether the subsequent abelian group in the long exact sequence vanishes.

\subsection{Mumford's construction}
A crucial ingredient in the construction of toroidal compactifications of the moduli stack $\mathcal{A}_g$ is a functor known as \textit{Mumford's construction}. Furthermore, this functor is needed in order to define arithmetic analogs of the Fourier and Fourier--Jacobi expansions of a Siegel modular form. We briefly sketch what the input and output of Mumford's construction are, following \cite{Mum72} and \cite[Chapter~III]{FC90}. 

We start by recalling a few preliminary definitions from \cite[Chapter~I]{FC90}. Let $S$ be any scheme and let $G\to S$ be a group scheme, with $e\colon S\to G$ its unit section. Let $\mathcal{L}$ be an invertible sheaf on $G$ which is \textit{rigidified along} $e$, that is, endowed with an isomorphism $\mathcal{O}_S\cong e^*\mathcal{L}$. For each subset $X\subseteq\{1,2,3\}$, let $m_X\colon G\times_S G\times_S G\to G$ be the morphism defined by the assignment $(g_1,g_2,g_3)\mapsto \sum_{i\in X}g_i$. Let
\[\Theta(\LL)\coloneqq\bigotimes_{X\subseteq\{1,2,3\}}m_X^*\LL^{(-1)^{\# X}}\]
and
\[\Lambda(\LL)\coloneqq m^*\LL\otimes p_1^*\LL^{-1}\otimes p_2^*\LL^{-1},\]
where $m,p_1,p_2\colon G\times_S G\to G$ denote the group law and the two projections, respectively. A \textit{cubical structure on} $\LL$ is an isomorphism $\Theta(\LL)\cong \mathcal{O}_{G\times_S G\times_S G}$ such that it induces on $\Lambda(\LL)$ the structure of a symmetric \textit{biextension} (see \cite{Mum69}) of $G\times_S G$ by the multiplicative group scheme $\mathbb{G}_{\textnormal{m}}$ over $S$. 

A \textit{torus} over $S$ is a commutative group scheme $T\to S$ which is, locally in the étale topology, $S$-isomorphic to finitely many copies of $\mathbb{G}_{\textnormal{m}}$. We call $T\to S$ a \textit{split torus} if it is (globally) $S$-isomorphic to finitely many copies of $\mathbb{G}_{\textnormal{m}}$. An \textit{abelian scheme} over $S$ is a smooth proper group scheme $A\to S$ with geometrically connected fibres. Abelian schemes are automatically commutative. Moreover, the well known theorem of the cube states that any invertible sheaf on an abelian scheme has a canonical cubical structure \cite[Theorem~I.1.3]{FC90}.

A \textit{semi-abelian scheme} over $S$ is a smooth separated commutative group scheme $G\to S$ with geometrically connected fibres, such that for all $s\in S$ the fibre $G_s$ is an extension of an abelian scheme $A_s$ by a torus $T_s$, i.e., the sequence $0\to T_s\to G_s\to A_s\to 0$ of group schemes is exact. The \textit{rank function} $r\colon S \to\mathbb{N}$ of a semi-abelian scheme $G\to S$ is defined by letting $r(s)$ denote the dimension of the torus part $T_s$ for each $s\in S$. It can be shown \cite[Remark I.2.4 and Corollary I.2.11]{FC90} that $r$ is upper semi-continuous and, if it is locally constant, then $G$ is globally an extension of an abelian scheme $A\to S$ by a torus $T\to S$, i.e., the sequence $0\to T\to G\to A\to 0$ of group schemes is exact.

Now, let $R$ be an integrally closed noetherian ring, which is complete with respect to a radical ideal $I=\sqrt{I}\subseteq R$. Further, let $K$ be the fraction field of~$R$, $S\coloneqq\Spec R$, and $S_0\coloneqq\Spec R/I$. Mumford's construction takes as input a semi-abelian scheme $\widetilde{G}\to S$ with constant rank function $r(s)=r$, and a homomorphism $\iota\colon \mathcal{Y}\to \widetilde{G}\otimes K$, where $\mathcal{Y}$ is an étale sheaf whose fibres are isomorphic to $\mathbb{Z}^r$. It then produces a semi-abelian scheme $G\to S$ as a kind of quotient of $\widetilde{G}$ by $\mathcal{Y}$, such that the generic fibre $G_\eta$ is an abelian scheme and $G\times_S S_0$ has constant rank function. 

A more refined variant of Mumford's construction takes additional data as input. Let $\widetilde{G}\to S$ and $\iota\colon \mathcal{Y}\to \widetilde{G}\otimes K$ as before. Since $\widetilde{G}\to S$ has constant rank function, it is globally an extension of an abelian scheme by a torus, namely, there is an exact sequence \[0\to T\to\widetilde{G}\xrightarrow{\varphi}A\to 0\] of group schemes over $S$, with $T$ a torus and $A$ an abelian scheme. Let $\widetilde{\LL}$ be an invertible sheaf on $\widetilde{G}$ with a cubical structure, such that $\widetilde{\LL}\cong\varphi^*\mathcal{M}$ for some ample invertible sheaf $\mathcal{M}$ on $A$ with a cubical structure. Furthermore, an action of $\mathcal{Y}$ on $\widetilde{\LL}_\eta$ over $\iota$ satisfying a certain positivity condition is given. Then, from these data, Mumford's construction produces a pair $(G,\LL)$ as a kind of quotient of $(\widetilde{G}, \widetilde{\LL})$ by $\mathcal{Y}$, where $G\to S$ is a semi-abelian scheme and $\LL$ is an invertible sheaf on $G$, such that the generic fibre $G_\eta$ is an abelian scheme and $\LL$ is ample over $G_\eta$.

Mumford's construction is carried out in \cite{Mum72} in the case that $\widetilde{G}\to S$ is a split torus and $\mathcal{Y}=Y$ is constant (being in fact a \textit{group of periods}, that is, a subgroup $Y\subset \widetilde{G}\otimes K$ isomorphic to $\mathbb{Z}^r$). The general case is handled in \cite[Chapter III]{FC90}. Defining domain and target categories appropriately, both variants of Mumford's construction turn out to be equivalences of categories, the quasi-inverse being described in \cite[Chapter II]{FC90}. Furthermore, Mumford's construction carries over to the case when the base is a formal scheme, by way of the theory of relative schemes over formal schemes, as developed in \cite{Hak72}.

\subsection{Arithmetic compactifications of \texorpdfstring{$\mathcal{A}_g$}{Ag}}\label{torcomp}
Let $\mathcal{A}_g$ denote the moduli stack of principally polarized abelian schemes $A\to S$ of relative dimension $g$. A toroidal compactification $\overline{\mathcal{A}}_g$ of $\mathcal{A}_g$ is constructed in \cite{FC90}, depending on the choice of certain combinatorial data, which is a smooth and proper algebraic stack over $\Spec\mathbb{Z}$ containing $\mathcal{A}_g$ as an open dense algebraic substack. We recall a few details regarding this construction.

Besides Mumford's construction, another important ingredient in the construction of $\overline{\mathcal{A}}_g$ is the theory of toric varieties, also known as torus embeddings. We would like to make a couple remarks here regarding the variance of terminology. According to \cite[Appendix A]{CLS11}, the first formal definition of a toric variety and of a fan (in French, \textit{éventail}) appeared in \cite{Dem70}. This is nowadays the most widely accepted terminology. However, Faltings--Chai draw their notation from \cite{KKMS73}, where a toric variety is called a \textit{torus embedding} and a fan is called a \textit{finite rational partial polyhedral cone decomposition}. Furthermore, Faltings--Chai make ad-hoc definitions that suit their purpose of constructing $\overline{\mathcal{A}}_g$, e.g., by identifying the character group of a $g(g+1)/2$-dimensional torus with the space of integral quadratic forms in $g$ variables, as well as adding an additional $\GL_g(\mathbb{Z})$-invariance condition to their fans. Since the present paper relies heavily on \cite{FC90}, we have decided to stick to their notation, which we now recall in detail.

Let $X$ be a free abelian group of rank $g$ and let $X_\mathbb{R}\coloneqq X\otimes_{\mathbb{Z}}\mathbb{R}$. We let $B(X)$ denote the space of symmetric bilinear forms $X\times X\to\mathbb{Z}$, and $B(X_{\mathbb{R}})$ the space of symmetric bilinear forms $X_{\mathbb{R}}\times X_{\mathbb{R}}\to\mathbb{R}$. The lattice $B(X)$ provides $B(X_{\mathbb{R}})$ with an integral structure. The \textit{radical of} a symmetric bilinear form $b\in B(X_{\mathbb{R}})$ is defined as the kernel of the map $X_{\mathbb{R}}\to\Hom(X_{\mathbb{R}},\mathbb{R})$ induced by $b$. We denote by $C(X)\subset B(X_{\mathbb{R}})$ the cone of positive semi-definite symmetric bilinear forms whose radicals are defined over $\mathbb{Q}$, and by $C(X)^\circ$ its interior, consisting of the positive definite symmetric bilinear forms whose radicals are defined over $\mathbb{Q}$. The action of $\GL(X)$ on $X$ induces actions of $\GL(X)$ on $B(X)$ and $B(X_{\mathbb{R}})$, which stabilize $C(X)$ and $C(X)^\circ$, preserving the integral structure.

A $\GL(X)$-\textit{admissible polyhedral cone decomposition of} $C(X)$ is a collection $\mathfrak{C}=\{\sigma_\alpha\}_{\alpha\in J}$ of subsets $\sigma_\alpha\subset B(X_{\mathbb{R}})$, satisfying the following three properties:

\begin{enumerate}
    \item [(i)] Each $\sigma_\alpha\in\mathfrak{C}$ is a strongly convex rational polyhedral cone, namely, \[\sigma_\alpha=\mathbb{R}_{>0}\cdot v_1+\cdots+\mathbb{R}_{>0}\cdot v_k\] is the set of positive real linear combinations of some symmetric bilinear forms $v_1,\dots,v_k\in B(X)\otimes_{\mathbb{Z}}\mathbb{Q}$.

    \item [(ii)] $C(X)=\coprod_{\alpha\in J}\sigma_\alpha$, and the closure of each cone in $\mathfrak{C}$ equals the disjoint union of finitely many cones in $\mathfrak{C}$.

    \item [(iii)] $\mathfrak{C}$ is invariant under the action of $\GL(X)$ on $B(X_{\mathbb{R}})$, and the restricted action of $\GL(X)$ on $\mathfrak{C}$ has finitely many orbits.
\end{enumerate}
Furthermore, we say that $\mathfrak{C}$ is \textit{smooth} if each $\sigma_\alpha\in\mathfrak{C}$ is a smooth cone, namely, if $\sigma_\alpha=\mathbb{R}_{>0}\cdot v_1+\cdots+\mathbb{R}_{>0}\cdot v_k$ where $v_1,\dots,v_k\in B(X)$ are part of a $\mathbb{Z}$-basis of $B(X)$.

A \textit{torus embedding} over a base scheme $S$ is an $E$-equivariant open immersion $E\hookrightarrow\overline{E}$ over $S$ of a split torus $E$ into a separated scheme $\overline{E}$, with an $E$-action extending the translation action of $E$, which is fibrewise an open dense immersion.

Over the complex numbers, the construction of a normal toric variety $X_\Sigma$ over $\mathbb{C}$ from a fan $\Sigma$ yields an equivalence of categories. In the present setting, over $\mathbb{Z}$, the construction of a functor from the category of finite rational partial polyhedral
cone decompositions to the category of torus embeddings over $\mathbb{Z}$ is extant, though it is no longer an equivalence \cite[Remark IV.2.6]{FC90}. Anyhow, a smooth $\GL(X)$-admissible polyhedral cone decomposition $\mathfrak{C}$ of $C(X)$ has an associated $g(g+1)/2$-dimensional torus embedding $E\hookrightarrow\overline{E}$ over $\Spec\mathbb{Z}$. The group of characters of $E$ is given by $S^2(X)$, the space of quadratic forms on $X$.

We now provide a sketch description of the local charts that make up $\overline{\mathcal{A}}_g$, which are adapted to the combinatorial data $\mathfrak{C}$. We start with the totally degenerate case, corresponding to cones $\sigma_\alpha\in\mathfrak{C}$ such that $\sigma_\alpha\subset C(X)^\circ$. Such a cone has an associated $E$-invariant affine open subscheme $E(\sigma_\alpha)=\Spec B_\alpha\subset\overline{E}$ and the closure $Z(\sigma_\alpha)$ of its associated $E$-orbit is a closed subscheme of $E(\sigma_\alpha)$, giving rise to an ideal $I_\alpha\subseteq B_\alpha$. Let $R_\alpha$ be the $I_\alpha$-adic completion of $B_\alpha$, let $K_\alpha$ be the fraction field of $R_\alpha$, and let $S_\alpha=\Spec R_\alpha$. If we let $\widetilde{G}$ denote the $g$-dimensional split torus over $S_\alpha$ with character group $X$, a subgroup of periods $Y\subset\widetilde{G}\otimes K_\alpha$ satisfying a certain positivity condition can be given thanks to the assumption that $\sigma_\alpha\subset C(X)^\circ$. Mumford's construction, as carried out in \cite{Mum72} in the totally degenerate case, now produces a semi-abelian scheme $ \prescript{\heartsuit}{}{G}_\alpha\to S_\alpha$ whose generic fibre is a principally polarized abelian scheme.

Let now $\sigma\in\mathfrak{C}$ be arbitrary. We have just seen that in case $\sigma_\alpha\subset C(X)^\circ$, one can construct a semi-abelian scheme which degenerates to a torus. In the general case, we then need to apply Mumford's construction for a general semi-abelian scheme $\widetilde{G}$ (with constant rank function) whose abelian part is not necessarily trivial, as done in \cite[Chapter III]{FC90}. It is then necessary to first define an appropriate base scheme, which takes some work as we will now see, and will in fact be a formal scheme.

 Each symmetric bilinear form $b\in\sigma$ induces a quadratic form $q(b)\colon X_\mathbb{R}\to\mathbb{R}$, given by the assignment $x\mapsto b(x,x)$. The cone \[\sigma^\vee\coloneqq\{q(b)\mid b\in\sigma\}\subseteq S^2(X_\mathbb{R})\] is the \textit{dual cone of} $\sigma$. We let $X\to X_\sigma$ denote the smallest quotient of $X$ such that $q(b)$ factors through $X_\mathbb{R}\to X_{\sigma,\mathbb{R}}\coloneqq X_\sigma\otimes_\mathbb{Z}\mathbb{R}$ for every $b\in\sigma$. Then, each bilinear form $b\in\sigma$ can be viewed as an element of $B(X_{\sigma,\mathbb{R}})$ and, furthermore, the subspace spanned by $\sigma$ is contained in $B(X_{\sigma,\mathbb{R}})$. Let $C(X_\sigma)\subset B(X_{\sigma,\mathbb{R}})$ be the cone of positive semi-definite symmetric bilinear forms whose radicals are defined over $\mathbb{Q}$, and let $C(X_\sigma)^\circ$ be its interior, consisting of the positive definite symmetric bilinear forms whose radicals are defined over $\mathbb{Q}$. Then, by polyhedral reduction theory \cite[Chapter II]{AMRT10}, the collection \[\mathfrak{C}_\sigma\coloneqq\{\sigma'\in\mathfrak{C}\mid\sigma'\subset B(X_{\sigma,\mathbb{R}})\}\] is a $\GL(X_\sigma)$-admissible polyhedral cone decomposition of $C(X_\sigma)$, having an associated torus embedding $E_\sigma\hookrightarrow\overline{E}_{\sigma}$. 

    We now fix an integer $0\leq l\leq g$. Let $\xi$, respectively $X_\xi$, denote the equivalence class of cones $\sigma\in\mathfrak{C}$, respectively quotients $X_\sigma$ of $X$, such that $\rk X_\sigma =l$, where $\rk$ denotes rank as a $\mathbb{Z}$-module. We let \begin{align*}
        &C(X_\xi)\coloneqq \bigcup_{\substack{\sigma\in\mathfrak{C} \\ \rk(X_\sigma)=l }} C(X_\sigma),  &\mathfrak{C}_\xi\coloneqq \bigcup_{\substack{\sigma\in\mathfrak{C} \\ \rk(X_\sigma)=l }} \mathfrak{C}_\sigma,
    \end{align*}
    and we note that $\mathfrak{C}_\xi$ is then a $\GL(X_\xi)$-admissible polyhedral cone decomposition of $C(X_\xi)$, having an associated $l(l+1)/2$-dimensional torus embedding $E_\xi\hookrightarrow\overline{E}_\xi$. We have $E_\xi=B(X_\xi)\otimes\mathbb{G}_{\textnormal{m}}$, and \[\overline{E}_\xi=\bigcup_{\sigma\in\mathfrak{C}_\xi}E_\xi(\sigma)=\coprod_{\sigma\in\mathfrak{C}_\xi}Z_\xi(\sigma),\]
    where $E_\xi(\sigma)$, respectively $Z_\xi(\sigma)$, denotes the affine $E_\xi$-invariant torus embedding, respectively the $E_\xi$-orbit, associated to the cone $\sigma$. We let $\mathfrak{C}^\circ_\xi\coloneqq\{\sigma\in\mathfrak{C}_\xi\mid \sigma\subset C(X_\xi)^\circ\}$ and further consider the locally closed subscheme \[Z_\xi\coloneqq\bigcup_{\sigma\in\mathfrak{C}^\circ_\xi} Z_\xi(\sigma)\subset \overline{E}_\xi.\] 
    Faltings--Chai then construct over $\mathbf{Hom}_\mathbb{Z}(X_\xi,\mathcal{X}_{g-l})$ (which is isomorphic to the fibred product of $l$ copies of the universal abelian scheme $\mathcal{X}_{g-l}$) a suitable $E_\xi$-torsor $\mathcal{E}_\xi$, and consider 
    \begin{align*}
        &\mathcal{E}(\sigma)\coloneqq\mathcal{E}_\xi\times^{E_\xi}E_\xi(\sigma), &\mathcal{Z}(\sigma)\coloneqq\mathcal{E}_\xi\times^{E_\xi}Z_\xi(\sigma)
    \end{align*} for each $\sigma\in\mathfrak{C}_\xi$, where the above products are \textit{contraction products} of torsors (also known as \textit{contracted products}, cf. \cite[Section 2.3]{Ems17}). We then note that \[\overline{\mathcal{E}}_\xi\coloneqq\mathcal{E}_\xi\times^{E_\xi}\overline{E}_\xi=\bigcup_{\sigma\in\mathfrak{C}_\xi}\mathcal{E}(\sigma)=\coprod_{\sigma\in\mathfrak{C}_\xi} \mathcal{Z}(\sigma),\]
       containing the locally closed subscheme \[\mathcal{Z}_\xi\coloneqq\mathcal{E}_\xi\times^{E_\xi} Z_\xi=\coprod_{\sigma\in\mathfrak{C}^\circ_\xi} \mathcal{Z}(\sigma).\] We define $\widehat{\overline{\mathcal{E}}}_\xi$ to be the formal completion of $\overline{\mathcal{E}}_\xi$ along $\mathcal{Z}_\xi$. This is the formal base scheme over which we will perform Mumford's construction. We also define $\widehat{\overline{\mathcal{E}}}_\sigma$ as the formal completion of $\overline{\mathcal{E}}_\xi$ along $\mathcal{Z}(\sigma)$.

       Let $A$ denote the pullback of $\mathcal{X}_{g-l}\to\mathcal{A}_{g-l}$ to $\mathbf{Hom}_\mathbb{Z}(X_\xi,\mathcal{X}_{g-l})$, and let $T_\xi$ denote the split torus with character group $X_\xi$. Then, over $\mathbf{Hom}_\mathbb{Z}(X_\xi,\mathcal{X}_{g-l})$, we have a tautological extension \[0\to T_\xi \to \widetilde{G}\to A\to 0,\] which we can then pullback to $\overline{\mathcal{E}}_\xi$ and further to $\widehat{\overline{\mathcal{E}}}_\xi$. Finally, one applies Mumford's construction here to obtain a semi-abelian scheme $\prescript{\heartsuit}{}{G}\to \widehat{\overline{\mathcal{E}}}_\xi$.
    
    With these local charts in hand, Faltings--Chai continue by algebraizing them and then gluing in the \'etale topology, culminating in the statement \cite[Theorem~IV.5.7]{FC90} that to a smooth $\GL(X)$-admissible polyhedral cone decomposition $\mathfrak{C}$ of $C(X)$ we can attach a smooth proper algebraic stack $\overline{\mathcal{A}}_g$ over $\Spec\mathbb{Z}$, having $\mathcal{A}_g$ as an open dense algebraic substack, such that the boundary $\overline{\mathcal{A}}_g\setminus\mathcal{A}_g$ is a relative Cartier divisor. The universal abelian scheme $\mathcal{X}_g\to\mathcal{A}_g$ extends to a semi-abelian scheme $\mathcal{G}_g\to\overline{\mathcal{A}}_g$. Furthermore, there is a stratification \[\overline{\mathcal{A}}_g=\coprod_{\sigma\in\mathfrak{C}/\GL(X)}\mathcal{Z}(\sigma),\] where $\sigma$ runs over a complete set of representatives of the finitely many $\GL(X)$-orbits of $\mathfrak{C}$. For example, the open stratum corresponding to $\sigma=\{0\}$ is $Z(\{0\})=\mathcal{A}_g$. Locally in the étale topology, $\overline{\mathcal{A}}_g$ is isomorphic to $\overline{E}$, where $E\hookrightarrow\overline{E}$ denotes the torus embedding associated to $\mathfrak{C}$, and this isomorphism preserves the respective stratifications.

    The formal completion of $\overline{\mathcal{A}}_g$ along $\mathcal{Z}(\sigma)$ is isomorphic to the formal algebraic stack $\widehat{\overline{\mathcal{E}}}_\sigma/\Gamma_\sigma$, where $\widehat{\overline{\mathcal{E}}}_\sigma$ denotes the formal completion of $\overline{\mathcal{E}}_\xi$ along $\mathcal{Z}(\sigma)$, $\Gamma_\sigma$ denotes the finite subgroup of $\GL(X_\sigma)$ stabilizing $\sigma$, and the quotient is taken as that of $\widehat{\overline{\mathcal{E}}}_\sigma$ as stack under the étale equivalence relation determined by $\Gamma_\sigma$. Similarly, if \[D_\xi\coloneqq\coprod_{\sigma\in \mathfrak{C}^\circ_\xi/\GL(X)}\mathcal{Z}(\sigma),\]  then \cite[Proposition~IV.5.9]{FC90} states that the formal completion of $\overline{\mathcal{A}}_g$ along $D_\xi$ is isomorphic to the stack quotient of $\widehat{\overline{\mathcal{E}}}_\xi$ by $\GL(X_\xi)$.

We end this subsection by leveraging the construction of an arithmetic toroidal compactification $\overline{\mathcal{A}}_g$ of $\mathcal{A}_g$, in order to introduce Siegel modular forms and the so-called minimal compactification. Let $\mathcal{G}_g\to\overline{\mathcal{A}}_g$ be the semi-abelian scheme extending the universal abelian scheme $\mathcal{X}_g\to\mathcal{A}_g$, whose unit section $\overline{\varepsilon}\colon \overline{\mathcal{A}}_g\to\mathcal{G}_g$ extends the unit section $\varepsilon\colon \mathcal{A}_g\to\mathcal{X}_g$. We consider the Hodge vector bundle $\mathbb{E}\coloneqq \varepsilon^*(\Omega_{\mathcal{X}_g/\mathcal{A}_g})$ on $\mathcal{A}_g$, which extends to $\overline{\mathcal{A}}_g$ as $\overline{\mathbb{E}}\coloneqq\overline{\varepsilon}^*(\Omega_{\mathcal{G}_g/\overline{\mathcal{A}}_g})$. We further consider the Hodge line bundles $\omega\coloneqq\det\mathbb{E}$ on $\mathcal{A}_g$ and $\overline{\omega}\coloneqq\det\overline{\mathbb{E}}$ on $\overline{\mathcal{A}}_g$. Let $M$ be any abelian group. A \textit{Siegel modular form of genus} $g$ \textit{and weight} $k$ \textit{with coefficients in} $M$ is defined to be an element of $H^0(\overline{\mathcal{A}}_g,\overline{\omega}^k\otimes_\mathbb{Z}M)$. 

There is an arithmetic version of Koecher's principle \cite[Proposition V.1.5]{FC90}, namely, the restriction map $H^0(\overline{\mathcal{A}}_g,\overline{\omega}^k\otimes_\mathbb{Z}M)\to H^0(\mathcal{A}_g,\omega^k\otimes_\mathbb{Z}M)$ is an isomorphism when $g\geq 2$. Over the field of complex numbers, we recover the classical $\mathbb{C}$-vector space of Siegel modular forms of genus $g$ and weight $k$ as $M_k(\Gamma_g)=H^0(\overline{\mathcal{A}}_g,\overline{\omega}^k\otimes_\mathbb{Z}\mathbb{C})$. From this point forward, we will abuse notation by writing $\omega$ for both $\det\mathbb{E}$ and $\det\overline{\mathbb{E}}$, and we will restrict our attention to the coefficient group $M=\mathbb{Z}$.

In \cite[Proposition V.2.1]{FC90} it is shown that, for some $m>0$, the tensor power $\omega^m$ is generated by its global sections over $\overline{\mathcal{A}}_g$, so it defines a morphism $\overline{\mathcal{A}}_g\to\mathbb{P}_\mathbb{Z}^n$ for some $n>0$. Then (see \cite[Theorem V.2.3]{FC90}) there is a Stein factorization
\[\begin{tikzcd}
\overline{\mathcal{A}}_g \arrow{r}{\pi}\arrow{dr}[swap]{} 
  & \mathcal{A}_g^*\arrow{d}\\ 
  &  \mathbb{P}_\mathbb{Z}^n
\end{tikzcd}\] such that $\pi\colon \overline{\mathcal{A}}_g\to\mathcal{A}_g^*$ has connected geometric fibres, $\pi_*\mathcal{O}_{\overline{\mathcal{A}}_g}=\mathcal{O}_{\mathcal{A}_g^*}$, and the morphism $\mathcal{A}_g^*\to\mathbb{P}_\mathbb{Z}^n$ is finite. This normal projective scheme \[\mathcal{A}_g^*=\Proj\big(\bigoplus_{k\geq 0}H^0(\overline{\mathcal{A}}_g,\omega^k)\big)\] is called the \textit{minimal compactification of} $\mathcal{A}_g$ (classically, over $\mathbb{C}$, it is known as the Satake--Baily--Borel compactification). It has a stratification by locally closed normal subschemes \[\mathcal{A}_g^*=\coprod_{0\leq i\leq g}[\mathcal{A}_i],\] where $[\mathcal{A}_i]$ is canonically isomorphic to the coarse moduli space of principally polarized abelian schemes of dimension $i$ over $\Spec\mathbb{Z}$. For each $0\leq i\leq g$, the closure of $[\mathcal{A}_i]$ is equal to $\coprod_{0\leq j\leq i}[\mathcal{A}_j]\cong \mathcal{A}_i^*$. Furthermore, after replacing $m$ by a suitable multiple, there is a very ample invertible sheaf $\LL$ on $\mathcal{A}_g^*$ such that $\pi^*\LL=\omega^m$.

\subsection{Fourier--Jacobi expansions}\label{ssfjexps}

We now explain how Mumford's construction allows us to define arithmetic analogues of the Fourier and Fourier--Jacobi expansions for an arithmetic Siegel modular form $f\in H^0(\overline{\mathcal{A}}_g,\omega^k)$. Recall that $\overline{\mathcal{A}}_g$ is associated to a smooth $\GL(X)$-admissible polyhedral cone decomposition $\mathfrak{C}$ of $C(X)$.

Let us begin with the Fourier expansion. Given $f\in H^0(\overline{\mathcal{A}}_g,\omega^k)$, its image in $H^0(\mathcal{A}_g,\omega^k)$ is a functor which, to any principally polarized abelian scheme $A\to\Spec R$ of relative dimension $g$ over a base ring $R$, associates an element $f(A)\in H^0(A,\omega^k_{A/R})$, where $\omega_{A/R}$ denotes the sheaf of translation invariant differential forms of highest degree on $A$. We will next evaluate $f$ at a specific abelian scheme obtained from Mumford's construction.

Let $\sigma\subset C(X)^\circ$ be a top dimensional cone. Since it is smooth, it is generated by a $\mathbb{Z}$-basis $s_1,\dots,s_r$ of the space $B(X)$ of symmetric billinear forms, where $r=g(g+1)/2$. Let $t_1,\dots,t_r$ be its dual basis in the space $S^2(X)$ of quadratic forms. We then consider the ring of formal power series $R'\coloneqq\mathbb{Z}[[t_1,\dots,t_r]]$. Let $K$ be the fraction field of $R'$ and let $\widetilde{G}$ be the $g$-dimensional split torus over $R'$ with character group $X$. The above choice of basis determines a positive definite bilinear form $b\colon X\otimes X\to K^{\times}$ via $b(\mu\otimes\nu)=\prod_i t_i^{s_i(\mu,\nu)}$, which in turn determines a period subgroup $Y\subset\widetilde{G}\otimes K$. We may thus apply Mumford's construction to produce a semi-abelian scheme $G$ over $R'$ with $\omega_{G/R'}\cong\omega_{\widetilde{G}/R'}\cong R'\otimes_{\mathbb{Z}}\bigwedge^g(X)$, and such that over the localization $R\coloneqq R'[1/t_1\cdots t_r]$ it is a principally polarized abelian scheme. Evaluating the functor $f\in H^0(\mathcal{A}_g,\omega^k)$ at this abelian scheme $G\to R$ thus yields an element $f(G)\in R\otimes_{\mathbb{Z}}\bigwedge^g(X)^{\otimes k}$, which we write as a formal power series \[f(G)=\sum_{\lambda\in S^2(X)}f(\lambda)\cdot\lambda\]  with coefficients $f(\lambda)\in \bigwedge^g(X)^{\otimes k}$. We call it the $q$-\textit{expansion of} $f$. It is shown in \cite[p. 139]{FC90} that $f(G)$ does not depend on the the choice of cone $\sigma\subset C(X)^\circ$.

Now, to define the more general Fourier--Jacobi expansion, recall that Mumford's construction can be applied to semi-abelian schemes $\widetilde{G}$ which are not necessarily tori, i.e., which have non-trivial abelian part. So let us now start with an arbitrary cone $\sigma\in\mathfrak{C}$ and let $l=\rk X_\sigma=\rk X_\xi$. 

As before, we let $\widehat{\overline{\mathcal{E}}}_\sigma$  and $\widehat{\overline{\mathcal{E}}}_\xi$ denote the formal completions of $\overline{\mathcal{E}}_\xi$ along $\mathcal{Z}(\sigma)$ and $\mathcal{Z}_\xi$, respectively. We consider the pullback of the tautological extension $0\to T_{\xi}\to\widetilde{G}\to A\to 0$ to $\widehat{\overline{\mathcal{E}}}_\sigma$ and apply Mumford's construction to produce a semi-abelian scheme $\prescript{\heartsuit}{}{G}\to \widehat{\overline{\mathcal{E}}}_\sigma$. The Fourier--Jacobi expansion of cogenus $l$ of a Siegel modular form $f\in H^0(\overline{\mathcal{A}}_g,\omega^k)$ is then defined by evaluating at $\prescript{\heartsuit}{}{G}\to \widehat{\overline{\mathcal{E}}}_\sigma$. More explicitly, let $\omega(\mathcal{X}_{g-l},\mathcal{A}_{g-l})$ denote the Hodge line bundle on $\mathcal{A}_{g-l}$, and let $\omega_\sigma$ and $\omega_\xi$ denote its pullbacks to $\widehat{\overline{\mathcal{E}}}_\sigma$ and to $\widehat{\overline{\mathcal{E}}}_\xi$, respectively. We then have an injective map

\[\FJ_{\sigma}\colon H^0(\overline{\mathcal{A}}_g,\omega^k)\to H^0(\widehat{\overline{\mathcal{E}}}_\sigma, \omega_\sigma^k)\]
given by
\[f\mapsto \sum_{\lambda\in S^2(X_\xi)}f(\lambda)\cdot\lambda,\]
 namely, assigning to each arithmetic Siegel modular form its Fourier--Jacobi expansion of cogenus $l$; we call it the \textit{Fourier--Jacobi homomorphism of cogenus l}. It is shown in \cite[p. 143]{FC90} that this construction does not depend on the choice of cone $\sigma\in\mathfrak{C}$ having $\rk X_\sigma=\rk X_\xi=l$, so we may regard $\FJ_\sigma$ as taking values in $H^0(\widehat{\overline{\mathcal{E}}}_\xi, \omega_\xi^k)$, and we thence denote it by $\FJ_\xi$. Furthermore, the coefficients $f(\lambda)$ are invariant under the natural $\GL(X_\xi)$-action.

\subsection{Principal level structures}\label{pls}

Let $A\to S$ be an abelian scheme of relative dimension $g$, let $n\geq 1$ be a positive integer, let $\mu_n$ denote the group of $n$-th roots of unity and let $\zeta_n\in\mu_n$ be a primitive root of unity. We denote the subgroup of $n$-torsion points of $A$ by $A[n]$, which we endow with a symplectic structure via the Weil pairing $A[n]\times A[n]\to \mu_n\cong\mathbb{Z}/n\mathbb{Z}$. We further endow the group $(\mathbb{Z}/n\mathbb{Z})^{2g}$ with the standard symplectic form 

\[ (\mathbb{Z}/n\mathbb{Z})^{2g}\times (\mathbb{Z}/n\mathbb{Z})^{2g}\to \mathbb{Z}/n\mathbb{Z}\]
given by the assignment 
\[(u,v)\times (z,w)\mapsto u{w}^t-v{z}^t.\]
A \textit{principal level n structure on} $A$ is an isomorphism $(\mathbb{Z}/n\mathbb{Z})^{2g}\cong A[n]$ which preserves the symplectic structures.

Let $\mathcal{A}_{g,n}$ denote the moduli stack of principally polarized abelian schemes $A\to S$ of relative dimension $g$ with a principal level $n$ structure. It is a smooth algebraic stack over $\Spec\mathbb{Z}[\zeta_n,1/n]$. We have $\mathcal{A}_{g,1}\cong\mathcal{A}_{g}$ and, whenever $m$ divides $n$, there is a natural morphism $\mathcal{A}_{g,n}\to\mathcal{A}_{g,m}$ which is finite and \'etale over $\Spec\mathbb{Z}[\zeta_n,1/n]$. There is a natural action of $\Sp_{2g}(\mathbb{Z}/n\mathbb{Z})$ on $\mathcal{A}_{g,n}$, defined by composing a principal level $n$ structure with elements of $\Sp_{2g}(\mathbb{Z}/n\mathbb{Z})$, and in fact we have \[\mathcal{A}_{g,n}/\Sp_{2g}(\mathbb{Z}/n\mathbb{Z})\cong\mathcal{A}_g\times_{\Spec\mathbb{Z}}\Spec\mathbb{Z}[\zeta_n,1/n].\]

The construction described in Subsection~\ref{torcomp} of toroidal compactifications of $\mathcal{A}_{g}$ carries over to $\mathcal{A}_{g,n}$. Indeed, let $\mathfrak{C}$ be a $\GL(X)$-admissible polyhedral cone decomposition  of $C(X)$, where $X$ is a free abelian group of rank $g$ and the integral structure of $C(X)\subset B(X_\mathbb{R})$ is now given by $n\cdot B(X)$. Then, \cite[Theorem~IV.6.7]{FC90} shows that we can attach to $\mathfrak{C}$ an algebraic stack $\overline{\mathcal{A}}_{g,n}$, proper over $\Spec\mathbb{Z}[\zeta_n,1/n]$, containing $\mathcal{A}_{g,n}$ as an open dense algebraic substack, and such that the boundary $\overline{\mathcal{A}}_{g,n}\setminus\mathcal{A}_{g,n}$ is a relative Cartier divisor. If $m$ divides $n$, the natural morphism $\mathcal{A}_{g,n}\to\mathcal{A}_{g,m}$ extends to a morphism $\overline{\mathcal{A}}_{g,n}\to\overline{\mathcal{A}}_{g,m}$ over $\Spec\mathbb{Z}[\zeta_n,1/n]$ which is equivariant with respect to the surjection $\Sp_{2g}(\mathbb{Z}/m\mathbb{Z})\to\Sp_{2g}(\mathbb{Z}/n\mathbb{Z})$. Furthermore, $\overline{\mathcal{A}}_{g,n}$ is the normalization of the natural map $\mathcal{A}_{g,n}\to\overline{\mathcal{A}}_g$ and, locally in the étale topology, it is isomorphic to the torus embedding defined by $\mathfrak{C}$. 

In order to describe the stratification of $\overline{\mathcal{A}}_{g,n}$, which is dictated by the cone decomposition $\mathfrak{C}$, we first need to introduce a few definitions. Recall that, for each $\sigma\in\mathfrak{C}$, we denote by $X_{\sigma}$ the smallest quotient of $X$ such that every quadratic form in $\sigma^\vee$ factors through $X_\mathbb{R}\to X_{\sigma,\mathbb{R}}$. Let \[\Sigma^*\coloneqq\coprod_{I}C(X_{\sigma}),\] where the index set $I$ is given by \[I\coloneqq\{\textnormal{surjections }\varphi\colon \mathbb{Z}^{2g}\to X_{\sigma}\mid(\ker\varphi)^\perp\subseteq\ker\varphi\},\] and where $\mathbb{Z}^{2g}$ is endowed with the standard symplectic form. The set $\Sigma^*$ inherits from the cone decomposition $\mathfrak{C}$ the structure of a simplicial complex.

We define an equivalence relation on $\Sigma^*$ as follows. Given cones $\sigma,\sigma'\in\mathfrak{C}$, a surjection $X_\sigma\to X_{\sigma'}$ and a cone $C(X_\sigma)$ indexed by a surjection $\mathbb{Z}^{2g}\to X_{\sigma}$, let $C(X_{\sigma'})$ be the cone indexed by the composite surjection $\mathbb{Z}^{2g}\to X_{\sigma}\to X_{\sigma'}$. Then, we identify $C(X_{\sigma'})$ with its image under the canonical injection $C(X_{\sigma'})\hookrightarrow C(X_{\sigma})$. Furthermore, we also identify the elements of $\Sigma^*$ lying in the same orbit of the natural action of $\Sp_{2g}(\mathbb{Z})$ on $\Sigma^*$. We then define the set $\Sigma_n$ as the quotient of $\Sigma^*$ under this equivalence relation. In general, $\Sigma_n$ is not necessarily a simplicial complex. However, if $n\geq 3$, a lemma of Serre (which dates back to Minkowski according to \cite{SZ96}) states that any algebraic integer in $\mathbb{Z}[\zeta_n]$ which is a unit and congruent to $1$ modulo $n$, is in fact equal to $1$. This implies that $\Sigma_n$ is a simplicial complex when $n\geq 3$, according to \cite[p. 129]{FC90}.

A stratification of $\overline{\mathcal{A}}_{g,n}$ is then shown in \cite[Theorem~IV.6.7]{FC90} to be given as \[\overline{\mathcal{A}}_{g,n}=\coprod_{\sigma\in\Sigma_n}\mathcal{Z}_n(\sigma),\] satisfying the property that, for every $\sigma\in\Sigma_n$, the formal completion of $\overline{\mathcal{A}}_{g,n}$ along the stratum $\mathcal{Z}_n(\sigma)$ is isomorphic to the formal algebraic stack $\widehat{\overline{\mathcal{E}}}_{\sigma,n}/\Gamma_{\sigma,n}$, where $\Gamma_{\sigma,n}$ is the stabilizer of $\sigma$ in $\ker(\GL(X_\sigma)\to\GL(X_\sigma/nX_\sigma))$ and $\widehat{\overline{\mathcal{E}}}_{\sigma,n}$ is defined in a manner analogous to that of $\widehat{\overline{\mathcal{E}}}_{\sigma}$ as explained in Subsection~\ref{torcomp}. In fact, if $n\geq 3$, then $\Gamma_{\sigma,n}$ is trivial by the above mentioned lemma of Serre--Minkowski, so, according to \cite[Corollary IV.6.9]{FC90}, in this case $\overline{\mathcal{A}}_{g,n}$ is an algebraic space.

The universal principally polarized abelian scheme $\mathcal{X}_{g,n}\to\mathcal{A}_{g,n}$ extends to a semi-abelian scheme $\mathcal{G}_{g,n}\to\overline{\mathcal{A}}_{g,n}$ which is canonically isomorphic to the pullback of the semi-abelian scheme $\mathcal{G}_g\to\overline{\mathcal{A}}_g$ by the morphism $\overline{\mathcal{A}}_{g,n}\to\overline{\mathcal{A}}_g$. We may thus consider the Hodge vector bundle $\mathbb{E}= \overline{\varepsilon}^*(\Omega_{\mathcal{G}_{g,n},\overline{\mathcal{A}}_{g,n}})$, where $\overline{\varepsilon}\colon \overline{\mathcal{A}}_{g,n}\to\mathcal{G}_{g,n}$ denotes the unit section, and the Hodge line bundle $\omega=\det\mathbb{E}$ on $\overline{\mathcal{A}}_{g,n}$. We define a \textit{Siegel modular form of genus g, weight k and level n} to be an element of $H^0(\overline{\mathcal{A}}_{g,n},\omega^k)$. 

Let us denote by $\Sp_{2g}(\mathbb{Z})(n)$ the kernel of the natural surjection $\Sp_{2g}(\mathbb{Z})\to\Sp_{2g}(\mathbb{Z}/n\mathbb{Z})$. The definition of the Fourier--Jacobi expansion of cogenus $l$ for elements in $H^0(\overline{\mathcal{A}}_g,\omega^k)$ naturally extends to elements in $H^0(\overline{\mathcal{A}}_{g,n},\omega^k)$, with the caveat that it now depends on the choice of a cusp of $\overline{\mathcal{A}}_{g,n}$ (see \cite[Proposition V.1.9]{FC90}). Here, a \textit{cusp of} $\overline{\mathcal{A}}_{g,n}$ is defined as a $\Sp_{2g}(\mathbb{Z})(n)$-orbit in the set of exact sequences \[0\to\Sigma_\mathbb{Z}\to\mathbb{Z}^{2g}\to\mathbb{Z}^{2g}/\Sigma_\mathbb{Z}\to 0,\] where $\mathbb{Z}^{2g}$ is endowed with the standard symplectic form and $\Sigma_\mathbb{Z}\subseteq\mathbb{Z}^{2g}$ is a co-torsion free subgroup with $\Sigma_\mathbb{Z}^\perp\subseteq\Sigma_\mathbb{Z}$. In particular, we may notice that there is a natural map from $\Sigma_n$ to the set of cusps of $\overline{\mathcal{A}}_{g,n}$.

From now on, let us assume that $n\geq 3$. In \cite[Theorem~2.5]{FC90}, it is shown that we also have a minimal compactification \[\mathcal{A}^*_{g,n}=\Proj\big(\bigoplus_{k\geq 0}H^0(\overline{\mathcal{A}}_{g,n},\omega^k)\big),\] which is a normal proper scheme of finite type over $\Spec\mathbb{Z}[\zeta_n,1/n]$, containing $\mathcal{A}_{g,n}$ as an open dense subscheme. The Hodge line bundle $\omega$ on $\mathcal{A}_{g,n}$ extends to an invertible sheaf $\LL$ on $\mathcal{A}^*_{g,n}$ relatively ample over $\Spec\mathbb{Z}$, and there is a canonical morphism $\pi\colon \overline{\mathcal{A}}_{g,n}\to\mathcal{A}^*_{g,n}$ such that $\pi^*\LL=\omega$. Furthermore, $\mathcal{A}^*_{g,n}$ has a stratification \[\mathcal{A}_{g,n}^*=\coprod_{0\leq i\leq g}\mathcal{A}_{i,n}\] by locally closed subschemes $\mathcal{A}_{i,n}$, geometrically normal and flat over $\Spec\mathbb{Z}$. For each $0\leq i\leq g$, the subscheme $\mathcal{A}_{i,n}$ is the fine moduli scheme of principally polarized abelian schemes of dimension $i$ over $\Spec\mathbb{Z}$ with a principal level $n$ structure, whose closure in $\mathcal{A}_{g,n}^*$ is isomorphic to $\mathcal{A}_{i,n}^*$.

A very nice feature of the case of principal level structures $n\geq 3$ is that, if we choose a certain type of fan $\mathfrak{C}$ to construct $\overline{\mathcal{A}}_{g,n}$, then $\overline{\mathcal{A}}_{g,n}$ can be exhibited as the normalization of a blow-up of $\mathcal{A}_{g,n}^*$ along a coherent sheaf of ideals. In particular, $\overline{\mathcal{A}}_{g,n}$ is a projective scheme in this case. Let us explain this in more detail.

Let $\mathfrak{C}$ be a smooth $\GL(X)$-admissible polyhedral cone decomposition of $C(X)$. An $\textit{invariant polarization function for}$ $\mathfrak{C}$ is a $\GL(X)$-invariant piecewise linear continuous function $\phi\colon C(X)\to\mathbb{R}_{\geq 0}$, such that \begin{enumerate}
    \item [(i)] $\phi(x)>0$ for $x\neq 0, \phi(tx)=t\phi(x)$ for all $x\in C(X), t\in\mathbb{R}_{\geq 0}$,
    \item [(ii)] $\phi$ is \textit{convex}, which in this context means that $\phi(x+y)\geq\phi(x)+\phi(y)$ for all $x,y\in C(X)$,
    \item [(iii)] $\phi|_\sigma$ is linear for every $\sigma\in\mathfrak{C}$, and
    \item [(iv)] $\phi$ takes integral values on $C(X)\cap B(X)$.
\end{enumerate} We say that $\mathfrak{C}$ is \textit{projective} if it admits an invariant polarization function. 

Let $\overline{\mathcal{A}}_{g,n}$ be the toroidal compactification associated to a smooth projective $\GL(X)$-admissible polyhedral cone decomposition of $C(X)$, with polarization function $\phi$, and let $K(n)$ denote the kernel of the map $\GL(X)\to\GL(X/nX)$. Let $D=\overline{\mathcal{A}}_{g,n}\setminus\mathcal{A}_{g,n}$ be the boundary divisor, with $D=\bigcup D_{\alpha}$ its decomposition into irreducible components, so that $D_{\alpha}$ corresponds to the $K(n)$-orbit of a $1$-dimensional cone $\sigma_{\alpha}\in\mathfrak{C}$. Let $c_\alpha\in\mathbb{Z}$ be the value of $\phi$ at the $\mathbb{Z}$-generator of $\sigma_\alpha\cap n\cdot B(X)$. Finally, consider the invertible sheaf of ideals $\J_{\phi}\coloneqq\mathcal{O}_{\overline{\mathcal{A}}_{g,n}}(-\sum_{\alpha}c_\alpha D_\alpha)$. We can now state a result from \cite{FC90}, which will be crucial for us in Subsection~\ref{secvalk}.

\begin{thm}\textnormal{\cite[Theorem V.5.8]{FC90}}\label{levblow}
    Suppose that $g\geq 2$ and $n\geq 3$. Then, there exists an integer $d\geq 1$ such that:
    \begin{enumerate}
        \item [(i)] $\overline{\mathcal{A}}_{g,n}$ is naturally isomorphic to the normalization of the blow-up of $\mathcal{A}_{g,n}^*$ along the coherent sheaf of ideals $\pi_*(\J^d_\phi)$.
        \item [(ii)] $\J^d_\phi=\pi^*\pi_*(\J^d_\phi)$.
    \end{enumerate}
\end{thm}


\subsection{Jacobi forms and formal Fourier--Jacobi series}\label{subsecjac}
Let $\varphi\colon \mathcal{X}_g\to \mathcal{A}_g$ denote the principally polarized universal abelian scheme, let $\mathcal{P}$ denote the Poincaré bundle on $\mathcal{X}_g\times_{\mathcal{A}_g} \mathcal{X}_g$, and let $\Delta\colon \mathcal{X}_g\to \mathcal{X}_g\times_{\mathcal{A}_g} \mathcal{X}_g$ denote the diagonal morphism. We then have over $\mathcal{X}_g$ a distinguished line bundle $\mathcal{L}\coloneqq\Delta^*\mathcal{P}$. Let $M$ be any abelian group. For each $k,m\in\mathbb{N}$ we consider the invertible sheaf $\mathcal{J}_{k,m}\coloneqq \mathcal{L}^m\otimes\varphi^*\omega^k$ on $\mathcal{X}_g$, and we define a \textit{Jacobi form of weight} $k$ \textit{and index} $m$ \textit{with coefficients in} $M$ to be an element of $H^0(\mathcal{X}_g, \mathcal{J}_{k,m}\otimes_{\mathbb{Z}}M)$. 

Over the field of complex numbers, we recover the classical $\mathbb{C}$-vector space of Jacobi forms as $J_{k,m}(\Gamma_g)=H^0(\mathcal{X}_g, \mathcal{J}_{k,m}\otimes_{\mathbb{Z}}\mathbb{C})$. From this point forward, we will restrict our attention to the coefficient group $M=\mathbb{Z}$. We further note here that in \cite[Theorem~2.12]{Kra95} it is shown that, for $k,m>0$, the invertible sheaf $\mathcal{J}_{k,m}$ on $\mathcal{X}_g$ is ample. 

Now, using Mumford's construction once again, we can define a $q$-expansion for arithmetic Jacobi forms, analogous to the Fourier expansion of classical Jacobi forms over $\mathbb{C}$. If $A\to\Spec R$ is a principally polarized abelian scheme of relative dimension $g$ over a base ring $R$, let $\mathcal{P}_A$ denote the Poincaré bundle on $A\times_{R}A$, let $\Delta_A\colon A\to A\times_{R} A$ denote the diagonal morphism and let $\mathcal{L}_A\coloneqq\Delta_A^*\mathcal{P}_A$. Each element $f\in H^0(\mathcal{X}_g, \mathcal{J}_{k,m})$ can be interpreted as a functor which, to the pair $(A,\mathcal{L}_A)$, associates an element \[f(A,\mathcal{L}_A)\in H^0(A,\omega_{A/R}^k\otimes\mathcal{L}_A^m),\] where $\omega_{A/R}$ denotes the sheaf of translation invariant differential forms of highest degree on $A$. We thus seek to evaluate $f$ at a specific pair $(A,\mathcal{L}_A)$ obtained from Mumford's construction.

Using the notation from our description in Subsection~\ref{ssfjexps} of the  $q$-expansion of an arithmetic Siegel modular form, recall the $g$-dimensional split torus $\widetilde{G}$ over $R'=\mathbb{Z}[[t_1,\dots,t_r]]$ with character group $X\cong\mathbb{Z}^g$, where $r=g(g+1)/2$. Let $\mathcal{L}_{\widetilde{G}}$ be the trivial cubical invertible sheaf on $\widetilde{G}$. Then, applying Mumford's construction to $(\widetilde{G}, \mathcal{L}_{\widetilde{G}})$ yields a pair $(G',\mathcal{L}_{G'})$, where $G'\to\Spec R'$ is a semi-abelian scheme with $\omega_{G/R'}\cong\omega_{\widetilde{G}/R'}\cong R'\otimes_{\mathbb{Z}}\bigwedge^g(X)$, such that the base change $G\coloneqq G'\times_{R'}\Spec R$ to the localization $R=R'[1/t_1\cdots t_r]$ is a principally polarized abelian scheme and $\mathcal{L}_{G}\coloneqq\mathcal{L}_{G'}\otimes R$ is the ample invertible sheaf on $G$ obtained by pulling back the Poincaré bundle on $G\times_R G$ via the diagonal morphism. We can expand $f(G,\mathcal{L}_G)$ according to the characters of $\widetilde{G}$ as \[f(G,\mathcal{L}_G)=\sum_{x\in X}f(G,x)\cdot x,\] with $f(G,x)\in R\otimes_{\mathbb{Z}}\bigwedge^g(X)^{\otimes k}$. By writing each of these coefficients as a power series \[f(G,x)=\sum_{\lambda\in S^2(X)}f(\lambda,x)\cdot\lambda,\] we finally obtain the so-called $q$-\textit{expansion} \[f(G,\mathcal{L}_G)=\sum_{(\lambda,x)\in S^2(X)\times X}f(\lambda,x)\cdot(\lambda,x)\] of the Jacobi form $f\in H^0(\mathcal{X}_g, \mathcal{J}_{k,m})$. It is shown in \cite[Remark 2.3]{Kra95} that the $q$-expansion of $f$ does not depend on the choice of basis $t_1,..,t_r$ of $S^2(X)$. We further note that the coefficients $f(\lambda,x)$ are invariant under the natural $\GL(X)$-action.

Let $\overline{\mathcal{A}}_g$ be the toroidal compactification associated to a $\GL(X)$-admissible polyhedral cone decomposition $\mathfrak{C}$ of $C(X)$. Let $\sigma\in\mathfrak{C}$ be a cone with $\rk X_\sigma=\rk X_{\xi}=1$, and let 
\[
    \FJ\colon H^0(\overline{\mathcal{A}}_g,\omega^k)\to H^0(\widehat{\overline{\mathcal{E}}}_\xi, \omega_\xi^k)
\] be the Fourier--Jacobi homomorphism of cogenus $1$, assigning to each Siegel modular form $f$ its Fourier--Jacobi expansion $\sum_{m\in\mathbb{Z}}f_m\chi^m$ of cogenus $1$. Recall that, here, $\omega_{\xi}$ denotes the pullback to $\widehat{\overline{\mathcal{E}}}_\xi$ of the Hodge line bundle on $\mathcal{A}_{g-1}$. 

It is shown in \cite[Lemma~3.4]{Kra22} that each $f_m$ is a Jacobi form, namely $f_m\in H^0(\mathcal{X}_{g-1}, \J_{k,m})$, so we have an injective map

\[H^0(\widehat{\overline{\mathcal{E}}}_\xi,\omega_\xi^k)\to\prod_{m\in\mathbb{N}}H^0(\mathcal{X}_{g-1}, \J_{k,m})\]
given by the assignment
\[\sum_{m\in\mathbb{N}}f_m\chi^m\mapsto (f_m)_{m\in\mathbb{N}}.\]
Each Jacobi form $f_m$ has a $q$-expansion \[\sum_{(\lambda,x)\in S^2(X')\times X'} f_m(\lambda,x),\] where $X'=\ker(X\to X_{\sigma})$. We define \[c(N)\coloneqq f_m(\lambda,x), \ \ \ N=\begin{pmatrix}
    \lambda &\frac{1}{2}x\\
    \frac{1}{2}x^t &m
\end{pmatrix}\in S^2(X),\] and we say that a tuple $(f_m)\in\prod_{m\in\mathbb{N}}H^0(\mathcal{X}_{g-1}, \J_{k,m})$ is a $\GL(X)$-\textit{invariant formal Fourier--Jacobi series} if the equality \[c(u^tNu)=\det(u)^kc(N)\] holds for every $N\in S^2(X)$ and $u\in\GL(X)$. We denote the subgroup of such $\GL(X)$-invariant tuples by $\big(\prod_{m\in\mathbb{N}}H^0(\mathcal{X}_{g-1}, \J_{k,m})\big)^{\GL(X)}$. 

According to \cite[Lemma~3.6]{Kra22} we have an isomorphism \[H^0(\widehat{\overline{\mathcal{E}}}_\xi,\omega_\xi^k)\cong\Big(\prod_{m\in\mathbb{N}}H^0(\mathcal{X}_{g-1}, \J_{k,m})\Big)^{\GL(X)},\] so we may write the map $\FJ$ as 
\[
\FJ\colon H^0(\overline{\mathcal{A}}_g,\omega^k)\to\Big(\prod_{m\in\mathbb{N}}H^0(\mathcal{X}_{g-1}, \J_{k,m})\Big)^{\GL(X)},\] 
assigning to each Siegel modular form $f$ the $\GL(X)$-invariant formal Fourier--Jacobi series $(f_m)_{m\in\mathbb{N}}$ given by the coefficients of its Fourier--Jacobi expansion. The main result in \cite{Kra22} then states  that $\FJ$ is surjective, thus showing that the modularity result of Bruinier--Raum carries over to the arithmetic setting over $\mathbb{Z}$, at least in the case of cogenus equal to $1$.

\subsection{The cohomological reformulation}\label{seccohref}
Let $\overline{\mathcal{A}}_{g}$ be a toroidal compactification of the moduli stack $\mathcal{A}_g$, and let $D$ be the relative Cartier divisor with reduced structure supported on the boundary $\overline{\mathcal{A}}_{g}\setminus\mathcal{A}_g$. Let also \begin{align*}
    \FJ\colon H^0(\overline{\mathcal{A}}_g,\omega^k)\to H^0(\widehat{\overline{\mathcal{E}}}_\xi, \omega_\xi^k)
\end{align*} be the Fourier--Jacobi homomorphism of cogenus $1$. Recall from the end of Subsection~\ref{torcomp} that the formal completion $\widehat{\overline{\mathcal{A}}}_{g}$ of $\overline{\mathcal{A}}_{g}$ along $D$ is isomorphic to the stack quotient of $\widehat{\overline{\mathcal{E}}}_\xi$ by $\GL(X_\xi)$, where $\rk X_\xi=1$. But in this case we have $\GL(X_\xi)=\GL(\mathbb{Z})=\{\pm 1\}$, so the action of $\GL(X_\xi)$ is trivial and $\widehat{\overline{\mathcal{A}}}_{g}$ is in fact isomorphic to $\widehat{\overline{\mathcal{E}}}_\xi$. In particular, we have \[H^0(\widehat{\overline{\mathcal{A}}}_{g},\widehat{\omega}^k)\cong H^0(\widehat{\overline{\mathcal{E}}}_\xi,\omega_\xi^k),\] where $\widehat{\omega}$ denotes the formal completion of $\omega$ along $D$. 

Now let $\mathcal{J}$ be a coherent sheaf of ideals supported on $\overline{\mathcal{A}}_{g}\setminus\mathcal{A}_g$; we call any such $\J$ a \textit{boundary sheaf}. Recall that the formal completion of $\overline{\mathcal{A}}_{g}$ along $D$ depends only on the support of $D$. Then, by Grothendieck's existence theorem on formal functions \cite[Theorem~4.1.5]{EGA3}, applied to the case of global sections, we have an isomorphism

 \[H^0(\widehat{\overline{\mathcal{A}}}_{g},\widehat{\omega}^k)\cong\varprojlim_m H^0(\overline{\mathcal{A}}_{g},\omega^k\otimes\bigslant{\OO_{\overline{\mathcal{A}}_g}}{\mathcal{J}^m}).\] Therefore, we may write the Fourier--Jacobi homomorphism as \[\FJ\colon H^0(\overline{\mathcal{A}}_g,\omega^k)\to \varprojlim_m H^0(\overline{\mathcal{A}}_{g},\omega^k\otimes\bigslant{\OO_{\overline{\mathcal{A}}_g}}{\mathcal{J}^m}).\] 
 
 We now observe that this natural map can be obtained from the short exact sequence of sheaves \[0\to\mathcal{J}^m\to\OO_{\overline{\mathcal{A}}_g}\to \bigslant{\OO_{\overline{\mathcal{A}}_g}}{\mathcal{J}^m}\to 0.\] Indeed, since $\omega$ is locally free, tensoring with $\omega^k$ gives again a short exact sequence of sheaves \[0\to\omega^k\otimes\mathcal{J}^m\to\omega^k\to \omega^k\otimes\bigslant{\OO_{\overline{\mathcal{A}}_g}}{\mathcal{J}^m}\to 0,\] which induces a long exact sequence \begin{align*}0&\to H^0(\overline{\mathcal{A}}_g,\omega^k\otimes\mathcal{J}^m)\to H^0(\overline{\mathcal{A}}_g,\omega^k)\to H^0(\overline{\mathcal{A}}_g, \omega^k\otimes\bigslant{\OO_{\overline{\mathcal{A}}_g}}{\mathcal{J}^m})\to\\
&\to  H^1(\overline{\mathcal{A}}_g,\omega^k\otimes\mathcal{J}^m)\to H^1(\overline{\mathcal{A}}_g,\omega^k)\to H^1(\overline{\mathcal{A}}_g, \omega^k\otimes\bigslant{\OO_{\overline{\mathcal{A}}_g}}{\mathcal{J}^m})\to\cdots\end{align*} in cohomology.

Next, we would like to apply the inverse limit functor with respect to $m$ to this long exact sequence. Taking an inverse limit is in general only left exact, but in this case our sequence will remain exact, in accordance with the following lemma.

\begin{lem}\label{invlimex}
Let \[0\to A_m\xrightarrow{f_m} B_m\xrightarrow{g_m} C_m\to 0\] be an exact sequence of inverse systems of abelian groups, indexed by $\mathbb{N}$. Suppose further that all the morphisms in the inverse system $\{B_m\}_{m\in\mathbb{N}}$ equal the identity morphism $\id\colon B\to B$ of some abelian group $B$. Then, the sequence \[0\to \varprojlim A_m\xrightarrow{f}\varprojlim B_m\xrightarrow{g}\varprojlim C_m\to 0\] of abelian groups is exact.
\end{lem}

\begin{proof}
Since taking inverse limit is left exact, we only need to prove that $g$ is surjective. Let $(c_m)_{m\in\mathbb{N}}\in\varprojlim C_m$. For each $m\in\mathbb{N}$ we consider the fiber \[D_m\coloneqq g_m^{-1}(c_m)\subseteq B_m,\] which is a non-empty subset of $B_m$ because $g_m$ is surjective. The morphisms in the inverse system $\{B_m\}_{m\in\mathbb{N}}$ restrict, giving $\{D_m\}_{m\in\mathbb{N}}$ the structure of an inverse system of sets. Any element in $\varprojlim D_m$ will then be a preimage of $(c_m)_{m\in\mathbb{N}}$ under $g$, so it suffices to show that $\varprojlim D_m$ is not empty. This is where our additional assumption comes into place: since $\{B_m\}_{m\in\mathbb{N}}$ consists solely of identity morphisms, so does $\{D_m\}_{m\in\mathbb{N}}$. In particular, $\varprojlim D_m=D_0\neq\varnothing$.
\end{proof}

Therefore, we have now located the Fourier--Jacobi homomorphism within the long exact sequence \begin{multline*}0\to \varprojlim_m H^0(\overline{\mathcal{A}}_g,\omega^k\otimes\mathcal{J}^m) \to H^0(\overline{\mathcal{A}}_g,\omega^k)\xrightarrow{\FJ} \varprojlim_m H^0(\overline{\mathcal{A}}_g, \omega^k\otimes\bigslant{\OO_{\overline{\mathcal{A}}_g}}{\mathcal{J}^m})\\
\to  \varprojlim_m H^1(\overline{\mathcal{A}}_g,\omega^k\otimes\mathcal{J}^m)\to H^1(\overline{\mathcal{A}}_g,\omega^k)\to\cdots\end{multline*} of abelian groups. As the main result in \cite{Kra22} shows that $\FJ$ is surjective, we are led to investigate vanishing conditions for the cohomology groups $H^1(\overline{\mathcal{A}}_g,\omega^k\otimes\mathcal{J}^m)$, and for their inverse limit \[\varprojlim_m H^1(\overline{\mathcal{A}}_g,\omega^k\otimes\mathcal{J}^m).\]

Similarly, our cohomological point of view passes to the case when principal level structures are considered. Indeed, let $n\geq 1$ be an integer and let $\mathcal{A}_{g,n}$ be the moduli stack of principally polarized abelian schemes $A\to S$ of relative dimension $g$, with a principal level $n$ structure. We continue denoting simply by $\omega$ the Hodge line bundle on $\overline{\mathcal{A}}_{g,n}$, and we let $\J_n$ denote the coherent sheaf of ideals supported on the boundary divisor $D_n\coloneqq\overline{\mathcal{A}}_{g,n}\setminus\mathcal{A}_{g,n}$ with reduced structure. We denote by $\FJ_n$ the homomorphism which appears in the long exact sequence
\begin{multline*}0\to \varprojlim_m H^0(\overline{\mathcal{A}}_{g,n},\omega^k\otimes\mathcal{J}_n^m) \to H^0(\overline{\mathcal{A}}_{g,n},\omega^k)\xrightarrow{\FJ_n} \varprojlim_m H^0\big(\overline{\mathcal{A}}_{g,n}, \omega^k\otimes\big(\bigslant{\OO_{\overline{\mathcal{A}}_{g,n}}}{\mathcal{J}_n^m}\big)\big)\\
\to  \varprojlim_m H^1(\overline{\mathcal{A}}_{g,n},\omega^k\otimes\mathcal{J}_n^m)\to H^1(\overline{\mathcal{A}}_{g,n},\omega^k)\to\cdots\end{multline*} of abelian groups. 

Now, in the level $n$ case, a Siegel modular form has one Fourier--Jacobi expansion for each cusp; we may observe that the map $\FJ_n$ provides us with all the Fourier--Jacobi expansions of cogenus $1$, as follows. Let $\mathcal{C}_n$ denote the set of cusps of $\overline{\mathcal{A}}_{g,n}$, that is, the set of $\Sp_{2g}(\mathbb{Z})(n)$-orbits of exact sequences \[0\to\Sigma_\mathbb{Z}\to\mathbb{Z}^{2g}\to\mathbb{Z}^{2g}/\Sigma_\mathbb{Z}\to 0\] where $\Sigma_\mathbb{Z}\subseteq\mathbb{Z}^{2g}$ is a co-torsion free subgroup with $\Sigma_\mathbb{Z}^\perp\subseteq\Sigma_\mathbb{Z}$. Recall from Subsection~\ref{pls} that $\overline{\mathcal{A}}_{g,n}$ has a stratification \[\overline{\mathcal{A}}_{g,n}=\coprod_{\sigma\in\Sigma_n}\mathcal{Z}_n(\sigma),\] and there is a natural surjection $\Sigma_n\to\mathcal{C}_n$. Let us also denote by $\mathcal{C}_n(1)$ the subset of $1$-dimensional cusps, that is, those which equal the image of some ray $\sigma\in\Sigma_n(1)$. Then, for each cusp $c\in\mathcal{C}_n(1)$, the union of strata \[D_c\coloneqq\coprod_{\substack{\sigma\in\Sigma_n(1)\\\sigma\mapsto c}} \mathcal{Z}_n(\sigma)\] parametrized by rays lying over $c$ is a divisor and, according to \cite[Corollary IV.6.11]{FC90}, the formal completion of $\overline{\mathcal{A}}_{g,n}$ along $D_c$ is isomorphic to a disjoint union of copies of $\widehat{\overline{\mathcal{E}}}_{\xi,n}$. Let us denote by $\J_c$ the ideal sheaf of $D_c$ with reduced structure. Since \[D=\coprod_{\sigma\in\Sigma_n(1)}\mathcal{Z}_n(\sigma)=\coprod_{c\in\mathcal{C}_n(1)}D_c,\] we then have \[H^0(\overline{\mathcal{A}}_{g,n}, \omega^k\otimes(\bigslant{\OO_{\overline{\mathcal{A}}_{g,n}}}{\mathcal{J}_n^m}))=\bigoplus_{c\in\mathcal{C}_n(1)}H^0\big(\overline{\mathcal{A}}_{g,n}, \omega^k\otimes\bigslant{\OO_{\overline{\mathcal{A}}_{g,n}}}{\mathcal{J}_c^m}\big)\] for every $m\geq 1$. If \[\FJ_c\colon H^0(\overline{\mathcal{A}}_{g,n},\omega^k)\to\varprojlim_m H^0\big(\overline{\mathcal{A}}_{g,n}, \omega^k\otimes\bigslant{\OO_{\overline{\mathcal{A}}_{g,n}}}{\mathcal{J}_c^m}\big)\] denotes the Fourier--Jacobi expansion of cogenus $1$ associated to a given cusp $c\in\mathcal{C}_n(1)$, we then see that the homomorphism \[\FJ_n\colon H^0(\overline{\mathcal{A}}_{g,n},\omega^k)\to\varprojlim_m H^0\big(\overline{\mathcal{A}}_{g,n}, \omega^k\otimes\bigslant{\OO_{\overline{\mathcal{A}}_{g,n}}}{\mathcal{J}_n^m}\big)\] decomposes as $\FJ_n=\bigoplus_{c\in\mathcal{C}_n(1)}\FJ_c$. 

In Theorem~\ref{surjeq}, we will prove that, when $n\geq 3$ and $k$ is sufficiently large, the homomorphism $\FJ_n$ is surjective if, and only if, the cohomological vanishing \[\varprojlim_m H^1(\overline{\mathcal{A}}_{g,n},\omega^k\otimes\mathcal{J}_n^m)=0\] holds. 
\section{Rational singularities and modularity}\label{chc}
    In the present section, we will make use of results from the theory of rational singularities. To our knowledge, these results are only available over algebraically closed fields of characteristic zero
    . Therefore, throughout this section we work over the field $\mathbb{C}$ of complex numbers.

   We begin in Subsection~\ref{projinv} by making a short note recalling how one determines the quotient of a projective scheme by the action of a finite group. In Subsection~\ref{ratsingg2}, following an idea suggested by R. Salvati Manni, we exhibit the minimal compactification $\mathcal{A}_2^*$ as the quotient of a smooth projective variety by a finite group of automorphisms. We then deduce that $\mathcal{A}_2^*$ has only rational singularities, and we derive some implications of this fact in connection with the modularity of formal Fourier--Jacobi series.

\subsection{Invariants of a projective scheme under a finite group}\label{projinv}

Let $X$ be a scheme and let $G$ be a group of automorphisms of $X$. By \textit{a quotient of} $X$ \textit{by} $G$ we mean a pair $(Y,p)$, where $Y$ is a scheme and $p\colon X\to Y$ is a morphism such that $p\circ g=p$ for all $g\in G$, satisfying the following universal property: for every morphism $f\colon X \to Z$ such that $f\circ g=f$ for all $g\in G$, there exists a unique morphism $\overline{f}\colon Y\to Z$ such that $\overline{f}\circ p=f$. When a quotient $(Y,p)$ of $X$ by $G$ exists, it follows from the definition that it is unique up to a unique isomorphism. We denote $(Y,p)$ by $X/G$ and call it \textit{the quotient of} $X$ \textit{by} $G$.

It is well known how to construct the quotient of an affine scheme by a finite group of automorphisms. Let $A$ be a commutative ring and let $G$ be a finite group of ring automorphisms of $A$. We denote by $A^G$ the subring of invariants of $A$ by the action of $G$, and we let $p\colon \Spec A \to\Spec A^G$ be the morphism induced by the inclusion $A^G\subseteq A$. Then, $(\Spec A^G,p)$ is the quotient of $\Spec A$ by $G$, and moreover the fibres of $p$ are precisely the orbits of the $G$-action on $\Spec A$. See, e.g., \cite[Proposition V.1.1]{SGA1} or \cite[Proposition~12.27]{GW20} for a proof.

The analogous statement for projective schemes holds too, and is also well known, but we have not been able to find a written proof. Most texts go directly from the affine case, to the case of a scheme with a $G$-action, such that every $G$-orbit is contained in an open affine subscheme (which includes all quasi-projective schemes), but they only show that the quotient exists. Therefore, we have thought it is worthwhile to pinpoint the projective case here.

\begin{lem}\label{projquot}
	Let $S$ be a graded ring and let $G$ be a finite group of graded ring automorphisms of $S$. Let $S^G$ be the subring of invariants of $S$ by the action of $G$, and let $p\colon \Proj S\to\Proj S^G$ be the induced morphism. Then, $(\Proj S^G,p)$ is the quotient of $\Proj S$ by $G$.
\end{lem}

\begin{proof}
Let $X\coloneqq \Proj S$ and let $S=\oplus_{d\geq 0}S_d$ be the grading for $S$. Then $S^G=\oplus_{d\geq 0}(S_d\cap S^G)$ is the induced grading on $S^G$. We let $S_+\coloneqq \oplus_{d>0}S_d$ and $S^G_+\coloneqq \oplus_{d>0}(S_d\cap S^G)$. For each $f\in S_+$ let $X_f\coloneqq \{\p\in\Proj S\mid f\notin\p\}$, so $\{X_f\}_{f\in S_+}$ is an affine open cover of $X$. We will show that, moreover, the subset $\{X_f\}_{f\in S^G_+}$ is an affine open cover of $X$ and $X_f$ is stable under the action of $G$ for every $f\in S^G_+$.

Indeed, let $f\in S^G_+$ and $\sigma\in G$. Then, for every $\p\in X_f$ we have $\sigma(f)=f\notin\p$, so $f\notin \sigma^{-1}(\p)$ and $\sigma^{-1}(\p)\in X_f$. Therefore $X_f$ is stable under the action of $G$ for every $f\in S^G_+$. Let us now show that \[X=\bigcup_{f\in S^G_+} X_f.\] Let $\p \in X$ and assume $\p\notin X_f$ for every $f\in S^G_+$, so $S^G_+\subseteq\p$. Let $f\in S_+$ and enumerate the finitely many elements of $G$ as $\sigma_1,...,\sigma_n$. We note that every elementary symmetric polynomial in $\sigma_1(f),...,\sigma_n(f)$ is invariant under $G$, and moreover has positive degree because $G$ respects the grading, so is an element of $S^G_+$. Since $S^G_+\subseteq\p$, we deduce that the reduction modulo $\p$ of the polynomial $F(t)=(t-\sigma_1(f))\cdots(t-\sigma_n(f))\in S[t]$ is equal to $t^n$. But $f$ is a root of $F$, so $f^n\in\p$, which implies $f\in\p$ since $\p$ is a prime ideal. Therefore $S+\subseteq\p$, contradicting the fact that $\p\in X=\Proj S$. Thus we conclude that $\{X_f\}_{f\in S^G_+}$ is a cover of $X$ consisting of affine, open $G$-stable subschemes.

Let $f\in S^G_+$ and consider the restriction of the action of $G$ to $X_f$. It is known that $X_f\cong\Spec S_{(f)}$ where $S_{(f)}$ is the subring of elements of degree 0 in the localized ring $S_f$, so we have quotient morphisms \[p_f\colon X_f\to\Spec S_{(f)}^G\] of affine schemes. Given any other $g\in S^G_+$ we note that, since the fibres of $p_f,p_g$ are the orbits of the $G$-action on $X_f,X_g$ respectively, we have \[p_f^{-1}p_f(X_f\cap X_g)=X_f\cap X_g=p_g^{-1}p_g(X_f\cap X_g).\] Now, by \cite[Corollaire V.1.4]{SGA1} both $p_f(X_f\cap X_g)$ and $p_g(X_f\cap X_g)$ are a quotient of $X_f\cap X_g$ by $G$, so they are canonically isomorphic. Then, by the proof of \cite[Proposition V.1.8]{SGA1} the glueing of the morphisms $p_f\colon X_f\to\Spec S_{(f)}^G$ constitutes a quotient of $X=\Proj S$ by $G$. But this glueing is canonically isomorphic to $\Proj S^G$, since $Y_f\coloneqq \{\p\in\Proj S^G\mid f\notin\p]\}\cong\Spec S^G_{(f)}$ for every $f\in S^G_+$ and $\{Y_f\}_{f\in S^G_+}$ is an open cover of $\Proj S^G$. Therefore, $(\Proj S)/G=\Proj S^G$.  
\end{proof}

\subsection{\texorpdfstring{$\mathcal{A}_2^*$}{A2*} has rational singularities}\label{ratsingg2}
It is shown in \cite{Tsuy86} and \cite{Tsuy88} that, if $g\geq 3$, then  $\mathcal{A}_g^*$ is not a Cohen--Macaulay scheme. If $\overline{\mathcal{A}}_g$ is any smooth toroidal compactification, it then follows from Kempf's criterion \cite[p. 50]{KKMS73} that the natural morphism $\pi\colon \overline{\mathcal{A}}_g\to\mathcal{A}_g^*$ is not a rational resolution. Therefore, if $g\geq 3$, the singularities of $\mathcal{A}_g^*$ are not rational. On the other hand, we will show in the present subsection that $\mathcal{A}_2^*$ is isomorphic to the quotient of Igusa's smooth modular variety $\mathcal{A}^*_2(2,4)$ by the action of a finite group. We will then conclude that $\mathcal{A}_2^*$ has only rational singularities, via Viehweg's result that quotient singularities are rational.

Let us first recall the concept of rational singularities.

\begin{defin}
    Let $Y$ be a reduced, irreducible, normal scheme of finite type over $\mathbb{C}$, and let $y\in Y$. 
    \begin{itemize}
        \item [(i)] We say that a morphism $f\colon X\to Y$ is a \textit{resolution} if $X$ is smooth and $f$ is proper, surjective and birational.
        \item [(ii)] We say that a resolution $f\colon X\to Y$ is \textit{rational} if $f_*\mathcal{O}_X\cong\mathcal{O}_Y$ and $R^if_*\mathcal{O}_X=0$ for all $i>0$.
        \item [(iii)] We say that $Y$ has a \textit{rational singularity at} $y$ if there exists an open neighborhood $U\ni y$ such that every resolution $X\to U$ is rational.
        \item [(iv)] We say that $Y$ has a \textit{quotient singularity at} $y$ if there exist an open neighborhood $U\ni y$ and a finite group $G$ acting on a smooth scheme $X$, such that $X/G$ exists and $X/G\cong U$.
        \item[(v)] We say that $Y$ has only rational singularities (resp. quotient singularities) if it has a rational singularity (resp. a quotient singularity) at every point.
    \end{itemize}
    \end{defin}
    According to \cite{Vie77} and the references therein, every regular point of a scheme is a rational singularity and, for $Y$ to have a rational singularity at $y\in U$, it suffices that one rational resolution $X\to U$ exists. According to \cite[Lemma~1]{Vie77}, $Y$ has only rational singularities if and only if every resolution $X\to Y$ is rational. Furthermore, according to \cite[Proposition~1]{Vie77} every quotient singularity is a rational singularity. 

Now, in order to introduce Igusa's modular variety $\mathcal{A}^*_2(2,4)$, we need to recall a few preliminary definitions. In Subsection~\ref{siegel} we defined Siegel modular forms with respect to the  Siegel modular group $\Gamma_g=\Sp_{2g}(\mathbb{Z})$. We now recall the straightforward generalization that includes the case of invariance with respect to normal subgroups of finite index in $\Gamma_g$.  

Let $\Gamma\trianglelefteq\Gamma_g$ be a normal subgroup of finite index. For each integer $k\in\mathbb{Z}$, we define the $\textit{weight-k operator}$ $\cdot$ on functions $f\colon \mathcal{H}_g\to\mathbb{C}$ by \[(f\cdot\gamma)(\tau)=\det(c\tau+d)^{-k}f(\gamma(\tau))\] for $\tau\in\mathcal{H}_g$ and $\gamma=\begin{psmallmatrix}
    a &b\\
    c &d
\end{psmallmatrix}\in\Gamma_g$. A \textit{Siegel modular form of genus g and weight k for} $\Gamma$ is thus defined as a holomorphic map $f\colon \mathcal{H}_g\to\mathbb{C}$, such that $f\cdot\gamma=f$ for every $\gamma\in\Gamma$ (plus the usual condition at infinity if $g=1$). They constitute a $\mathbb{C}$-vector space, which we denote by $M_k(\Gamma)$, and $M_\bullet(\Gamma)\coloneqq\oplus_{k\in\mathbb{Z}}M_k(\Gamma)$ is a graded ring.

\begin{lem}\label{action}
The weight-$k$ operator defines a right linear action of the quotient group $\Gamma_g/\Gamma$ on the vector space $M_k(\Gamma)$. The subspace of invariants of this action is given by $M_k(\Gamma_g)$.
\end{lem}

\begin{proof}
Let \[
\gamma=\begin{pmatrix}
a &b \\ c &d
\end{pmatrix}, \ \gamma'=\begin{pmatrix}
a' &b' \\ c' &d'
\end{pmatrix}\in\Gamma_g.\]
\noindent We have
\begin{align*}
(c'\tau+d')(c\gamma'(\tau)+d)&=(c'\tau+d')(c(a'\tau+b')(c'\tau+d')^{-1}+d)\\
&=c(a'\tau+b')+d(c'\tau+d')\\
&=(ca'+dc')\tau+cb'+dd'
\end{align*} 
for every $\tau\in\mathcal{H}_g$, so
\begin{align*}
(f\cdot\gamma)\cdot\gamma'(\tau)&=\det(c'\tau+d')^{-k}(f\cdot\gamma)(\gamma'(\tau))\\
&=\det(c'\tau+d')^{-k}\det(c\gamma'(\tau)+d)^{-k}f(\gamma(\gamma'(\tau)))\\
&=\det((c'\tau+d')(c\gamma'(\tau)+d))^{-k}f(\gamma(\gamma'(\tau)))\\
&=\det((ca'+dc')\tau+cb'+dd')^{-k}f(\gamma(\gamma'(\tau)))\\
&=f\cdot(\gamma\gamma')(\tau)
\end{align*}

\noindent for every $f\in M_k(\Gamma)$, $\tau\in\mathcal{H}_g$.

Given $\alpha\in\Gamma_g$ and $f\in M_k(\Gamma)$, let us show that $f\cdot\alpha\in M_k(\Gamma)$. We have:
\begin{align*}
(f\cdot\alpha)\cdot\gamma&=f\cdot \alpha \ \ \ \forall \gamma\in\Gamma\\
\iff \ \ \ \ \  f\cdot(\alpha\gamma)&=f\cdot\alpha \ \ \ \forall \gamma\in\Gamma\\
\iff f\cdot(\alpha\gamma\alpha^{-1})&=f \ \ \ \ \ \ \ \forall \gamma\in\Gamma\\
\Longleftarrow \ \ \ \ \ \ \ \ \alpha\gamma\alpha^{-1} &\in \Gamma \ \ \ \ \ \ \ \forall \gamma\in\Gamma,
\end{align*} and the latter holds because $\Gamma\unlhd\Gamma_g$. Hence $f\cdot\alpha\in M_k(\Gamma)$, and the action
\[
M_k(\Gamma)\times\Gamma_g\to M_k(\Gamma)
\] given by the assignment $(f,\alpha)\mapsto f\cdot\alpha$ is well defined. Since $f\cdot\gamma=f$ for every $\gamma\in\Gamma$, it descends to a well defined, linear action \[
	M_k(\Gamma)\times \bigslant{\Gamma_g}{\Gamma}\to M_k(\Gamma).
\] By definition, the subspace of invariants of this action is given by $M_k(\Gamma_g)$.
\end{proof}

Let $n\geq 1$ be an integer. We define the \textit{principal congruence subgroup of level} $n$ as
\[\Gamma_g(n)\coloneqq \{\gamma\in\Gamma_g \ | \ \gamma\equiv 1_{2g}\Mod{n}\}\leq\Gamma_g.\] A subgroup of $\Gamma_g$ is called a \textit{congruence subgroup of level} $n$ if it contains $\Gamma_g(n)$. For example, we consider the Igusa groups \[\Gamma_g(n,2n)\coloneqq\big\{\gamma=\begin{pmatrix}
a &b \\ c &d
\end{pmatrix}\in\Gamma_g \ | \ \gamma\equiv 1_{2g} \Mod{n}, \  (a^tb)_0\equiv (c^td)_0\equiv 0 \Mod{2n} \big\}\] as introduced in \cite{Igu64a}. Here, if $s$ is a square matrix, the vector formed by its diagonal coefficients is denoted by $(s)_0$. We have inclusions \[\Gamma_g(2n)\subset\Gamma_g(n,2n)\subset\Gamma_g(n),\] so $\Gamma_g(n,2n)$ is a congruence subgroup of level $2n$. If $n$ is even, then $\Gamma_g(n,2n)$ is a normal subgroup of $\Gamma_g$ and its index in $\Gamma_g(n)$ is $2^{2g}$ \cite[Lemma~1]{Igu64a}. Hence, by Lemma~\ref{action}, the quotient group $\Gamma_g/\Gamma_g(n,2n)$ acts on $M_k(\Gamma_g(n,2n))$. This action extends linearly to the full graded ring $M_\bullet(\Gamma_g(n,2n))$.

    \begin{thm}\label{maina2}
        The minimal compactification $\mathcal{A}_2^*$ has only rational singularities.
    \end{thm}

    \begin{proof}
       The finite group \[G\coloneqq \bigslant{\Gamma_2}{\Gamma_2(2,4)}\] acts on the graded ring $S\coloneqq M(\Gamma_2(2,4))$ with subring of invariants $S^G=M(\Gamma_2)$. If $\mathcal{A}^*_2(2,4)\coloneqq \Proj M_\bullet(\Gamma_2(2,4))$, we then have \[\bigslant{\mathcal{A}^*_2(2,4)}{G}=\mathcal{A}^*_2 \] by Lemma~\ref{projquot}. Furthermore, since $\mathcal{A}^*_2(2,4)$ is smooth \cite[p. 397]{Igu64b}, we deduce that $\mathcal{A}^*_2$ has only quotient singularities. But every quotient singularity is a rational singularity \cite[Proposition~1]{Vie77}, so $\mathcal{A}^*_2$ has only rational singularities.
    \end{proof}

    \begin{rmk}
        Let $\mathcal{A}_{g,n}$ denote the moduli space of $g$-dimensional principally polarized abelian varieties with a principal level $n$ structure. Since $\Gamma_2(2,4)\subseteq\Gamma_2(2)$, the same argument as in the proof of Theorem~\ref{maina2} shows that $\mathcal{A}^*_{2,2}=\Proj M_\bullet(\Gamma_2(2))$ has only quotient singularities. On the other hand, if $n\geq 3$, it is no longer the case that $\Gamma_2(2,4)\subseteq\Gamma_2(n)$. In fact, since $\Gamma_2(n)$ has no elements of finite order (other than $\pm 1_{2g}$) for $n\geq 3$, it turns out that $\mathcal{A}^*_{2,n}=\Proj M_\bullet(\Gamma_2(n))$ has quite bad singularities (see \cite[Corollary~2]{Igu64b} for a precise statement).
    \end{rmk}

\begin{cor}\label{vanw2}
    Let  $\overline{\mathcal{A}}_2$ be a smooth toroidal compactification of $\mathcal{A}_2$. There exists an integer $k_0\geq 0$ such that, if $k\geq k_0$, then \[H^i(\overline{\mathcal{A}}_2,\omega^k)=0\] for all $i>0$.
\end{cor}

\begin{proof}
    Since $\mathcal{A}^*_2$ has only rational singularities, every resolution $X\to\mathcal{A}^*_2$ is rational \cite[Lemma~1]{Vie77}. In particular, the natural morphism $\pi\colon \overline{\mathcal{A}}_2\to\mathcal{A}^*_2$ is a rational resolution, so \[R^i\pi_*\mathcal{O}_{\overline{\mathcal{A}}_2}=0\] for all $i>0$. Let $\mathcal{L}$ denote the ample invertible sheaf on $\mathcal{A}^*_2$ satisfying $\pi^*\mathcal{L}=\omega$. Then, by the projection formula and the usual degenerate Leray spectral sequence argument, we have
    \begin{align*}
 H^i(\overline{\mathcal{A}}_2, \omega^{k})&\cong H^i(\mathcal{A}^*_2, \pi_*\omega^{k})\\
 &\cong H^i(\mathcal{A}^*_2, \pi_*(\OO_{\overline{\A}_2})\otimes \LL^k)\\
 &\cong H^i(\mathcal{A}^*_2, \OO_{\A^*_2}\otimes \LL^k)\\
 &\cong H^i(\mathcal{A}^*_2, \LL^k)
\end{align*} for all $i>0$. The claim thus follows from Serre's vanishing theorem \cite[Proposition~III.5.3]{Har77}.
\end{proof}

We end this section by observing that Corollary~\ref{vanw2} allows us to show, in the case $g=2$, that the surjectivity of the Fourier--Jacobi homomorphism is in fact equivalent to the cohomological vanishing suggested in Subsection~\ref{seccohref}, at least for sufficiently large weights.

\begin{cor}\label{surjeq2}
    Let  $\overline{\mathcal{A}}_2$ be a smooth toroidal compactification of $\mathcal{A}_2$ and let $\J$ be a coherent sheaf of ideals supported on $\overline{\mathcal{A}}_2\setminus \mathcal{A}_2$. Then, there exists an integer $k_0\geq 0$ such that, for every $k\geq k_0$, the Fourier--Jacobi homomorphism 
    \[\FJ\colon H^0(\overline{\mathcal{A}}_{2},\omega^k)\to \varprojlim_m H^0(\overline{\mathcal{A}}_{2}, \omega^k\otimes \bigslant{\OO_{\overline{\mathcal{A}}_{2}}}{\mathcal{J}^m})\] is surjective if, and only if, the cohomological vanishing \[\varprojlim_m H^1\big(\overline{\mathcal{A}}_{2},\omega^k\otimes\mathcal{J}^m\big)=0\] holds. 
\end{cor}

\begin{proof}
   Let $k_0$ be as in Corollary~\ref{vanw2}, and let us assume that $\FJ$ is surjective. Let $k\geq k_0$. From the long exact sequence \begin{align*}0\to \varprojlim_m H^0(\overline{\mathcal{A}}_2,\omega^k\otimes\mathcal{J}^m) &\to H^0(\overline{\mathcal{A}}_2,\omega^{k})\xrightarrow{\FJ} \varprojlim_m H^0\big(\overline{\mathcal{A}}_2,\omega^k\otimes \bigslant{\OO_{\overline{\mathcal{A}}_{2}}}{\mathcal{J}^m}\big)\\
&\xrightarrow{\psi}  \varprojlim_m H^1(\overline{\mathcal{A}}_2,\omega^k\otimes\mathcal{J}^m)\to H^1(\overline{\mathcal{A}}_2,\omega^{k})\to\cdots,\end{align*} we deduce that the homomorphism $\psi$ is both surjective and equal to $0$, since $H^1(\overline{\mathcal{A}}_2,\omega^k)=0$ and $\FJ$ is surjective. Therefore $\varprojlim_m H^1(\overline{\mathcal{A}}_{2},\omega^k\otimes\mathcal{J}^m)=0$.
\end{proof}

We shall be able to generalize Corollary~\ref{surjeq2} to arbitrary $g\geq 2$ in Theorem~\ref{surjeq}, where we furthermore work over $\mathbb{Z}$, but with the caveat that a principal level $n\geq 3$ structure needs to be added.
\section{Cohomological vanishing in the arithmetic setting}\label{chz}

The main results of this section are Theorem~\ref{largek} and Theorem~\ref{surjeq}. The former is a cohomological vanishing result for sufficiently large weights, while the latter establishes the equivalence between the surjectivity of the Fourier--Jacobi homomorphism in level $n\geq 3$, and the vanishing of the subsequent abelian group in the corresponding long exact sequence, again for sufficiently large weights. In both cases, we derive our proofs from certain positivity properties of the Hodge line bundle, and of a sheaf of ideals supported on the boundary of a given toroidal compactification of $\mathcal{A}_{g,n}$. 

Throughout this section, we work over $\mathbb{Z}$.


\subsection{Vanishing for sufficiently large weights}\label{secvalk}

In the present subsection, our main tool is going to be Theorem~\ref{levblow}, so we need to introduce a principal level structure. Let $g\geq 2$ and $n \geq 3$ be integers, and let $\mathcal{A}_{g,n}$ denote the moduli stack of principally polarized abelian schemes $A\to S$ of relative dimension $g$, with a principal level $n$ structure. Let $X$ be a free abelian group of rank $g$, and let $C(X)$ denote the cone of positive semi-definite symmetric bilinear forms $X_{\mathbb{R}}\times X_{\mathbb{R}}\to\mathbb{R}$ whose radicals are defined over $\mathbb{Q}$. Let $\overline{\mathcal{A}}_{g,n}$ be the toroidal compactification of $\mathcal{A}_{g,n}$ associated to a smooth projective $\GL(X)$-admissible polyhedral cone decomposition $\mathfrak{C}$ of $C(X)$, with polarization function $\phi$. 

Let $\pi\colon \overline{\mathcal{A}}_{g,n}\to \mathcal{A}^*_{g,n}$ be the natural morphism to the minimal compactification, which has the property that $\omega=\pi^*\LL$, where $\omega$ denotes the Hodge line bundle on $\overline{\mathcal{A}}_{g,n}$, and $\LL$ is an ample invertible sheaf on $\mathcal{A}^*_{g,n}$. Furthermore, let $\J_\phi$ be the invertible coherent sheaf of ideals supported on $\overline{\mathcal{A}}_{g,n}\setminus \mathcal{A}_{g,n}$, whose multiplicities at each irreducible component of $\overline{\mathcal{A}}_{g,n}\setminus \mathcal{A}_{g,n}$ are dictated by the polarization function $\phi$, as in Theorem~\ref{levblow}. Then, Theorem~\ref{levblow} states that there exists an integer $d\geq 1$, such that the morphism $\pi\colon \overline{\mathcal{A}}_{g,n}\to\mathcal{A}_{g,n}^*$ factorizes as \[\overline{\mathcal{A}}_{g,n}\xrightarrow{\psi}\widetilde{\mathcal{A}}_{g,n}^*\xrightarrow{\varphi}\mathcal{A}_{g,n}^*,\] where $\varphi$ is the blow-up of $\mathcal{A}_{g,n}^*$ along the coherent sheaf of ideals $\pi_*(\J_\phi^d)$, and $\psi$ is the normalization morphism. Furthermore, we have $\pi^*\pi_*(\J_\phi^d)=\J_\phi^d$. 

In order to simplify our notation, we let $\J\coloneqq\J_\phi^d$, which is itself an invertible coherent sheaf of ideals supported on $\overline{\mathcal{A}}_{g,n}\setminus \mathcal{A}_{g,n}$. We begin our investigation of vanishing conditions for the various cohomology groups $H^1(\overline{\mathcal{A}}_{g,n},\omega^k\otimes\mathcal{J}^m)$, as the parameters $k$ and $m$ vary, with the following auxiliary positivity result for the boundary sheaf $\J$.

\begin{lem}\label{relamp}
    The sheaf of ideals $\J=\J_\phi^d$ is ample relative to the morphism $\pi\colon \overline{\mathcal{A}}_{g,n}\to \mathcal{A}_{g,n}^*$.
\end{lem}

\begin{proof}
    Recall that $\pi$ factorizes as \[\overline{\mathcal{A}}_{g,n}\xrightarrow{\psi}\widetilde{\mathcal{A}}_{g,n}^*\xrightarrow{\varphi}\mathcal{A}_{g,n}^*,\] where $\varphi$ is the blow-up of $\mathcal{A}_{g,n}^*$ along $\pi_*\J$, and $\psi$ is the normalization morphism.

    Let $\mathcal{\J'}\coloneqq \varphi^{-1}(\pi_*\J)\cdot\OO_{\widetilde{\mathcal{A}}_{g,n}^*}$ be the inverse image ideal sheaf of $\pi_*\J$ under the morphism $\varphi$. By definition, $\mathcal{J}'$ is given by the image of the natural map $\varphi^*\pi_*\J\to \OO_{\widetilde{\mathcal{A}}_{g,n}^*}$ induced by the inclusion $\pi_*\J\hookrightarrow\OO_{\mathcal{A}_{g,n}^*}$. We apply the right exact functor $\psi^*$ to the diagram \[\varphi^*\pi_*\J\twoheadrightarrow \mathcal{J}'\hookrightarrow\OO_{\widetilde{\mathcal{A}}_{g,n}^*}\] and obtain 
 \begin{equation}\label{eq1}
 \psi^*\varphi^*\pi_*(\J)\twoheadrightarrow\psi^*\mathcal{J}'\to\OO_{\overline{\mathcal{A}}_{g,n}}.\end{equation} Since $\psi^*\varphi^*\pi_*(\J)=\pi^*\pi_*(\J)=\J$ is a sheaf of ideals in $\OO_{\overline{\mathcal{A}}_{g,n}}$, the composition of the two natural morphisms in (\ref{eq1}) is (naturally isomorphic to) the inclusion $\J\hookrightarrow \OO_{\overline{\mathcal{A}}_{g,n}}$. Therefore, the morphism $\J\to\psi^*\mathcal{\J'}$ is an isomorphism.

 Now, as shown in the proof of \cite[Proposition~II.7.13]{Har77}, we have $\mathcal{J}'=\OO_{\widetilde{\mathcal{A}}_{g,n}^*}(1)$, so $\J'$ is very ample relative to $\varphi$. Furthermore, normalization maps are integral; in particular, they are quasi-affine, so $\J\cong\psi^*\J'$ is ample relative to $\pi$ by \href{https://stacks.math.columbia.edu/tag/0892}{[Stacks, Lemma~29.37.7]}. 
\end{proof}

We are now able to establish a cohomological vanishing result for a sufficiently large weight~$k$.

\begin{thm}\label{largek}
    Let us fix integers $g\geq 2$ and $n\geq 3$. Then, there exist an integer $m_0>0$ and a function $F\colon \mathbb{N}\to\mathbb{N}$ such that, if $m\geq m_0$ and $k\geq F(m)$, then \[H^i(\overline{\A}_{g,n},\omega^{k}\otimes\J^m)=0\] for all $i>0$.
\end{thm}

\begin{proof}
    Since $\J$ is ample relative to $\pi$, by the relative version of Serre's vanishing theorem \href{https://stacks.math.columbia.edu/tag/02O1}{[Stacks, Lemma~30.16.2]} there is an integer $m_0>0$ such that, if $m\geq m_0$, then \[R^i\pi_*(\J^{m})=0\] for all $i>0$. Recall that $\omega=\pi^*\LL$, where $\LL$ is an ample invertible sheaf on $\mathcal{A}^*_{g,n}$. By the projection formula, for every $k\geq 0$ and every $m\geq m_0$, we have \begin{align*}
        R^i\pi_*(\omega^k\otimes\J^m)&\cong \LL^k\otimes R^i\pi_*(\J^m)\\
        &=0
    \end{align*} for all $i>0$. Hence, as a degenerate case of the Leray spectral sequence, for every $k\geq 0$ and every $m\geq m_0$, we have isomorphisms
\begin{align*}H^i(\overline{\A}_{g,n}, \omega^{k}\otimes\J^{m})&\cong H^i(\A^*_{g,n}, \pi_*(\omega^k\otimes\J^{m}))\\
&\cong H^i(\A^*_{g,n},\LL^k\otimes\pi_*(\J^{m}))
\end{align*} for all $i\geq 0$. Since $\LL$ is ample, the claim now follows from Serre's vanishing theorem \cite[Proposition~III.5.3]{Har77}.
\end{proof}

\subsection{Cohomological characterization of modularity}\label{characz}

We were led to investigate vanishing results for the sheaf cohomology groups $H^1(\overline{\mathcal{A}}_{g,n},\omega^k\otimes\mathcal{J}^m)$ by the observation that, from the exactness of the sequence \begin{multline*}0\to \varprojlim_m H^0(\overline{\mathcal{A}}_{g,n},\omega^k\otimes\mathcal{J}^m) \to H^0(\overline{\mathcal{A}}_{g,n},\omega^k)\xrightarrow{\FJ_n} \varprojlim_m H^0\big(\overline{\mathcal{A}}_{g,n}, \omega^k\otimes\bigslant{\OO_{\overline{\mathcal{A}}_{g,n}}}{\mathcal{J}^m}\big)\\
\to  \varprojlim_m H^1(\overline{\mathcal{A}}_{g,n},\omega^k\otimes\mathcal{J}^m)\to H^1(\overline{\mathcal{A}}_{g,n},\omega^k)\to\cdots,\end{multline*} it follows that the vanishing \[\varprojlim_m H^1(\overline{\mathcal{A}}_{g,n},\omega^k\otimes\mathcal{J}^m)=0\] would imply the surjectivity of $\FJ_n$. We can now prove, for sufficiently large $k$, that this is in fact a necessary condition.

\begin{thm}\label{surjeq}
    Let us fix integers $g\geq 2$ and $n\geq 3$. There exists an integer $k_0\geq 0$ such that, if $k\geq k_0$, then the homomorphism 
    \[\FJ_n\colon H^0(\overline{\mathcal{A}}_{g,n},\omega^k)\to \varprojlim_m H^0\big(\overline{\mathcal{A}}_{g,n}, \omega^k\otimes\bigslant{\OO_{\overline{\mathcal{A}}_{g,n}}}{\mathcal{J}^m}\big)\] is surjective if, and only if, the cohomological vanishing \[\varprojlim_m H^1(\overline{\mathcal{A}}_{g,n},\omega^k\otimes\mathcal{J}^m)=0\] holds. 
\end{thm}

\begin{proof}
    Let us assume that $\FJ_n$ is surjective. Applying the pushforward functor $\pi_*$ to the short exact sequence \[ 0\to\omega^k\otimes\mathcal{J}^m\to \omega^k\to\omega^k\otimes\bigslant{\OO_{\overline{\mathcal{A}}_{g,n}}}{\mathcal{J}^m}\to 0\] of sheaves on $\overline{\mathcal{A}}_{g,n}$, and using the projection formula, produces the long exact sequence 
    \begin{equation}\label{longsheaves}
        0\to\LL^k\otimes\pi_*(\mathcal{J}^{m})\to\LL^k\to \LL^k\otimes\pi_*\big(\bigslant{\OO_{\overline{\mathcal{A}}_{g,n}}}{\mathcal{J}^{m}}\big)\to \LL^k\otimes R^1\pi_*(\mathcal{J}^{m})\to\cdots
    \end{equation} 
    of sheaves on $\mathcal{A}_{g,n}^*$. According to Lemma~\ref{relamp}, the invertible sheaf $\J$ is ample relative to the morphism $\pi$, so the relative version of Serre's vanishing theorem \href{https://stacks.math.columbia.edu/tag/02O1}{[Stacks, Lemma~30.16.2]} implies that there is an integer $m_0>0$ such that, if $m\geq m_0$, then $R^1\pi_*(\J^{m})=0$. Thus, for $m\geq m_0$, the sequence \eqref{longsheaves} is a short exact sequence, and hence induces a long exact sequence
    \begin{multline*}
0\to H^0({\mathcal{A}}^*_{g,n},\LL^k\otimes\pi_*(\mathcal{J}^{m})) \xrightarrow{} H^0({\mathcal{A}}^*_{g,n},\LL^k)\xrightarrow{} H^0\big({\mathcal{A}}^*_{g,n}, \LL^k\otimes\pi_*\big(\bigslant{\OO_{\overline{\mathcal{A}}_{g,n}}}{\mathcal{J}^{m}}\big)\big)\\ 
\xrightarrow{}   H^1({\mathcal{A}}^*_{g,n},\LL^k\otimes\pi_*(\mathcal{J}^{m}))\xrightarrow{} H^1({\mathcal{A}}^*_{g,n},\LL^k)\to\cdots
\end{multline*} of cohomology groups. By Lemma~\ref{invlimex}, if we now apply the inverse limit functor, then the resulting sequence
\begin{multline*}
0\to \varprojlim_m H^0({\mathcal{A}}^*_{g,n},\LL^k\otimes\pi_*(\mathcal{J}^{m})) \xrightarrow{\beta_1} H^0({\mathcal{A}}^*_{g,n},\LL^k)\xrightarrow{\beta_2}\\\varprojlim_m H^0({\mathcal{A}}^*_{g,n}, \LL^k\otimes\pi_*(\bigslant{\OO_{\overline{\mathcal{A}}_{g,n}}}{\mathcal{J}^{m}})) 
\xrightarrow{\beta_3}  \varprojlim_m H^1({\mathcal{A}}^*_{g,n},\LL^k\otimes\pi_*(\mathcal{J}^{m}))\\\xrightarrow{\beta_4} H^1({\mathcal{A}}^*_{g,n},\LL^k)\to\cdots
\end{multline*} remains exact. 

For any $\OO_{\overline{\mathcal{A}}_{g,n}}$-module $\mathcal{F}$, the functoriality of cohomology and the adjunction between $\pi^*$ and $\pi_*$ give natural maps $H^i(\mathcal{A}^*_{g,n},\pi_*\mathcal{F})\to H^i(\overline{\mathcal{A}}_{g,n}, \pi^*\pi_*\mathcal{F})\to H^i(\overline{\mathcal{A}}_{g,n}, \mathcal{F})$ for all $i\geq 0$. Therefore, we have a natural map between long exact sequences, as portrayed in the following commutative diagram.

    \[\begin{tikzcd}
 0\arrow{d}  &0\arrow{d}  \\
 \varprojlim\limits_m H^0({\mathcal{A}}^*_{g,n},\LL^k\otimes\pi_*(\mathcal{J}^{m})) \arrow{d}{\beta_1} \arrow{r}{\gamma_1} &\varprojlim\limits_m H^0(\overline{\mathcal{A}}_{g,n},\omega^k\otimes\mathcal{J}^m)\arrow{d}{\alpha_1} \\
 H^0({\mathcal{A}}^*_{g,n},\LL^k)\arrow{d}{\beta_2}\arrow{r}{\gamma_2} &H^0(\overline{\mathcal{A}}_{g,n},\omega^k)\arrow{d}{\FJ_n} \\
 \varprojlim\limits_m H^0({\mathcal{A}}^*_{g,n}, \LL^k\otimes\pi_*(\bigslant{\OO_{\overline{\mathcal{A}}_{g,n}}}{\mathcal{J}^{m}}))\arrow{d}{\beta_3} \arrow{r}{\gamma_3} &\varprojlim\limits_m H^0(\overline{\mathcal{A}}_{g,n}, \omega^k\otimes\bigslant{\OO_{\overline{\mathcal{A}}_{g,n}}}{\mathcal{J}^m})\arrow{d}{\alpha_3}\\
 \varprojlim\limits_m H^1({\mathcal{A}}^*_{g,n},\LL^k\otimes \pi_*(\mathcal{J}^{m}))\arrow{r}{\gamma_4}\arrow{d}{\beta_4} &\varprojlim\limits_m H^1(\overline{\mathcal{A}}_{g,n},\omega^k\otimes\mathcal{J}^m)\arrow{d}{\alpha_4}\\
 H^1({\mathcal{A}}^*_{g,n},\LL^k)\arrow{d}\arrow{r}{\gamma_5} &H^1(\overline{\mathcal{A}}_{g,n},\omega^k)\arrow{d}\\
 \vdots &\vdots
\end{tikzcd}\]

We start our analysis of this diagram by noticing that the homomorphism $\gamma_4$ is an isomorphism. Indeed, as in the proof of Theorem~\ref{largek}, since $R^i\pi_*(\J^{m})=0$ for all $i>0$ if $m\geq m_0$, we have isomorphisms \begin{align*}H^1(\overline{\A}_{g,n}, \omega^k\otimes\mathcal{J}^{m})&\cong H^1(\A^*_{g,n}, \pi_*(\omega^k\otimes\mathcal{J}^{m}))\\
&\cong H^1(\A^*_{g,n},\LL^k\otimes\pi_*(\J^{m}))
\end{align*} for $m\geq m_0$, and hence the induced map of inverse limits $\gamma_4$ is also an isomorphism.

Now, since $\LL$ is ample, by Serre's vanishing theorem \cite[Proposition~III.5.3]{Har77} there exists an integer $k_0\geq 0$ such that $H^1(\mathcal{A}_{g,n}^*,\LL^k)=0$ for all $k\geq k_0$. Let us then suppose that $k\geq k_0$, so the map $\beta_3$ is surjective. Furthermore, since $\FJ_n$ is assumed to be surjective, we also have $\alpha_3=0$. Then, the equality \[\gamma_4\circ\beta_3=\alpha_3\circ\gamma_3=0\] holds. But $\gamma_4$ is an isomorphism, so we have $\beta_3=0$. Therefore, since the homomorphism $\beta_3$ is both surjective and equal to zero, we have \[\varprojlim\limits_m H^1(\overline{\mathcal{A}}_{g,n},\omega^k\otimes\mathcal{J}^{m})\cong\varprojlim\limits_m H^1({\mathcal{A}}^*_{g,n},\LL^k\otimes\pi_*(\mathcal{J}^{m}))=0.\]
This completes the proof.
\end{proof}


\end{document}